\date{\today} 
\documentclass[10pt]{article}
\usepackage{hyperref}
\usepackage{amssymb, amsmath}
\usepackage{stmaryrd}
\usepackage{latexsym}
\usepackage{abstract}

\usepackage{enumitem}
\setlist[enumerate,1]{label={\rm (\alph*)}}
\setlist[enumerate,2]{label={\rm (\roman*)}}

\newlist{inlinelist}{enumerate*}{1}
\setlist*[inlinelist,1]{%
  label=({\rm \arabic*}),
}

\usepackage{amsthm}

\font\tengoth=eufm10 at 10pt
\font\sevengoth=eufm7 at 6pt
\newfam\gothfam
\textfont\gothfam=\tengoth
\scriptfont\gothfam=\sevengoth

\newcommand{\g}{{\mathfrak g}}

\newcommand{\h}{{\mathfrak h}}

\newcommand{\ff}{{\mathfrak f}}
\newcommand{\fg}{{\mathfrak g}}
\newcommand{\fh}{{\mathfrak h}}

\newcommand{\fk}{{\mathfrak k}}

\newcommand{\fn}{{\mathfrak n}}

\newcommand{\fq}{{\mathfrak q}}

\newcommand{\1}{\mathbf{1}}

\newcommand{\cD}{\mathcal{D}}
\newcommand{\cE}{\mathcal{E}}

\newcommand{\cH}{\mathcal{H}}

\newcommand{\cL}{\mathcal{L}}

\newcommand{\cN}{\mathcal{N}}

\newcommand{\cS}{\mathcal{S}}

\newcommand{\cV}{\mathcal{V}}

\newcommand{\derat}[1]{\frac{d}{dt}\big\vert_{t = #1}}
\newcommand{\bigderat}[1]{\frac{d}{dt}\Big\vert_{t = #1}}

\newcommand{\dd}{{\tt d}}

\newcommand{\N}{{\mathbb N}}

\newcommand{\R}{{\mathbb R}}
\newcommand{\C}{{\mathbb C}}

\newcommand{\bP}{{\mathbb P}}

\renewcommand{\hat}{\widehat}

\renewcommand{\tilde}{\widetilde}
\renewcommand{\L}{\mathop{\bf L{}}\nolimits}

\newcommand{\SL}{\mathop{{\rm SL}}\nolimits}

\newcommand{\OO}{\mathop{\rm O{}}\nolimits}

\newcommand{\U}{\mathop{\rm U{}}\nolimits}

\newcommand{\fsl} {\mathop{{\mathfrak{sl} }}\nolimits}

\newcommand{\ad}{\mathop{{\rm ad}}\nolimits}
\newcommand{\Ad}{\mathop{{\rm Ad}}\nolimits}

\renewcommand{\Re}{\mathop{{\rm Re}}\nolimits}
\renewcommand{\Im}{\mathop{{\rm Im}}\nolimits}

\newcommand{\Hom}{\mathop{{\rm Hom}}\nolimits}

\newcommand{\Aut}{\mathop{{\rm Aut}}\nolimits}

\newcommand{\End}{\mathop{{\rm End}}\nolimits}
\newcommand{\id}{\mathop{{\rm id}}\nolimits}

\newcommand{\im}{\mathop{{\rm im}}\nolimits}

\newcommand{\spec}{\mathop{{\rm spec}}\nolimits}

\newcommand{\supp}{\mathop{{\rm supp}}\nolimits}

\newcommand{\oline}{\overline}

\newcommand{\la}{\langle}
\newcommand{\ra}{\rangle}

\newcommand{\spann}{{\rm span}}

\newcommand{\ssssarr}{\hbox to 15pt{\rightarrowfill}}
\newcommand{\sssarr}{\hbox to 20pt{\rightarrowfill}}
\newcommand{\ssarr}{\hbox to 30pt{\rightarrowfill}}
\newcommand{\sarr}{\hbox to 40pt{\rightarrowfill}}
\newcommand{\arr}{\hbox to 60pt{\rightarrowfill}}
\newcommand{\larr}{\hbox to 60pt{\leftarrowfill}}
\newcommand{\Arr}{\hbox to 80pt{\rightarrowfill}}

\def\theoremname{Theorem}
\def\propositionname{Proposition}
\def\corollaryname{Corollary}
\def\lemmaname{Lemma}
\def\remarkname{Remark}
\def\conjecturename{Conjecture} 

\def\definitionname{Definition}
\def\exercisename{Exercise}
\def\examplename{Example}
\def\examplesname{Examples}
\def\problemname{Problem}
\def\problemsname{Problems}

\def\bemerkname{Bemerkung}
\def\aufgname{Aufgabe}

\def\@thmcounter#1{\noexpand\arabic{#1}}
\def\@thmcountersep{}
\def\@begintheorem#1#2{\it \trivlist \item[\hskip 
\labelsep{\bf #1\ #2.\quad}]}
\def\@opargbegintheorem#1#2#3{\it \trivlist
      \item[\hskip \labelsep{\bf #1\ #2.\quad{\rm #3}}]}
\makeatother
\newtheorem{theor}{\theoremname}[section]
\newtheorem{propo}[theor]{\propositionname}
\newtheorem{coro}[theor]{\corollaryname}
\newtheorem{lemm}[theor]{\lemmaname}

\newenvironment{thm}{\begin{theor}\it}{\end{theor}}
\newenvironment{theorem}{\begin{theor}\it}{\end{theor}}

\newenvironment{prop}{\begin{propo}\it}{\end{propo}}

\newenvironment{cor}{\begin{coro}\it}{\end{coro}}

\newenvironment{lem}{\begin{lemm}\it}{\end{lemm}}
\newenvironment{lemma}{\begin{lemm}\it}{\end{lemm}}

\newtheorem{rema}[theor]{\remarkname}

\newenvironment{rem}{\begin{rema}\rm}{\end{rema}}

\newtheorem{stepnow}[theor]{}

\newtheorem{defin}[theor]{\definitionname} 

\newenvironment{definition}{\begin{defin}\rm}{\end{defin}}

\newtheorem{exerc}{\exercisename}[section]

\newtheorem{exa}[theor]{\examplename}

\newenvironment{example}{\begin{exa}\rm}{\end{exa}}

\newtheorem{exas}[theor]{\examplesname}

\newenvironment{examples}{\begin{exas}\rm}{\end{exas}}

\newtheorem{conj}[theor]{\conjecturename}

\newtheorem{pro}[theor]{\problemname}

\newtheorem{prs}[theor]{\problemsname}

\newcommand{\pmat}[1]{\begin{pmatrix} #1 \end{pmatrix}}

\newcommand{\Rmnum}[1]{\uppercase\expandafter{\romannumeral#1}}

\newcommand{\p}{\partial}

\renewcommand{\phi}{\varphi}

\newcommand{\Stand}{\mathop{{\rm Stand}}\nolimits}

\newcommand{\reg}{\mathrm{reg}}

\makeatletter
\def\blfootnote{\xdef\@thefnmark{}\@footnotetext}
\makeatother

\begin{document}
\title{Analytic extensions of representations of \(*\)-subsemigroups without polar decomposition} 

\author{Daniel Oeh} 

\maketitle


\begin{abstract}
  Let \((G,\tau)\) be a finite-dimensional Lie group with an involutive automorphism \(\tau\) of \(G\) and let \(\g = \fh \oplus \fq\) be its corresponding Lie algebra decomposition.
  We show that every non-degenerate strongly continuous representation on a complex Hilbert space \(\cH\) of an open \(*\)-subsemigroup \(S \subset G\), where \(s^* = \tau(s)^{-1}\), has an analytic extension to a strongly continuous unitary representation of the 1-connected Lie group \(G_1^c\) with Lie algebra \([\fq,\fq] \oplus i\fq\).
  
  We further examine the minimal conditions under which an analytic extension to the 1-connected Lie group \(G^c\) with Lie algebra \(\fh \oplus i\fq\) exists. This result generalizes the L\"uscher--Mack Theorem and the extensions of the L\"uscher--Mack Theorem for \(*\)-subsemigroups satisfying \(S = S(G^\tau)_0\) by Merigon, Neeb, and \'Olafsson.

  Finally, we prove that non-degenerate strongly continuous representations of certain \(*\)-subsemigroups \(S\) can even be extended to representations of a generalized version of an Olshanski semigroup.
\end{abstract}

\section{Introduction}
\label{sec:intro}

In the context of unitary representation theory, the problem of analytic extensions can be stated as follows: Let \((G,\tau)\) be a pair consisting of a Lie group \(G\) and an involutive automorphism \(\tau\) on \(G\). By decomposing the Lie algebra \(\g\) of \(G\) into the \((+1)\)-eigenspace \(\fh\) and the \((-1)\)-eigenspace \(\fq\) of \(\L(\tau)\), we obtain a decomposition \(\g = \fh \oplus \fq\). Let \(S\) be an open subsemigroup of \(G\) which is invariant under the operation \(g^* := \tau(g)^{-1}, g \in G,\) and let \(\pi: S \rightarrow B(\cH)\) be a strongly continuous \(*\)-representation of \(S\) on a complex Hilbert space \(\cH\) by bounded operators. Then the goal is to find a strongly continuous representation \(\pi^c : G^c \rightarrow \U(\cH)\) of the 1-connected Lie group \(G^c\) with Lie algebra \(\g^c = \fh \oplus i\fq\) which is uniquely determined by an analytic continuation property.

A well-known example of this problem are strongly continuous self-adjoint one-parameter semigroups \(\pi : (\R_{> 0}, +) \rightarrow B(\cH)\). Here, we have \(G = \R\) and \(\tau = -\id_\R\). In this case, the infinitesimal generator \(A\) of \(\pi\) is selfadjoint and, by functional calculus, the representation \(\pi^c(it) := e^{itA}, t \in \R,\) of \(G^c = i\R\) is an analytic extension of \(\pi\) (cf.\ Example \ref{ex:one-dim}).

More generally, the following theorem is proven in \cite{LM75}, known as the L\"uscher--Mack Theorem: Let \(H \subset G\) be the integral subgroup with Lie algebra \(\fh\) and let \(C \subset \fq\) be a non-empty open convex cone which is invariant under the adjoint action of \(H\).
Consider the \(*\)-semigroup \(\Gamma(C)\) generated by \(H\exp(C)\).
Then every contraction representation \(\pi\) of \(\Gamma(C)\) on a complex Hilbert space can be analytically continued to a strongly continuous unitary representation \(\pi^c\) of \(G^c\) in the sense that the infinitesimal generators of the one-parameter (semi)-groups of elements in \(\fh + C\) coincide up to the obvious multiplication with \(i\).

Since the proof of the L\"uscher--Mack Theorem in \cite{LM75} relies on the existence of coordinates of the second kind (cf.\ \cite[Lem.\ 9.2.6]{HN12}), it only works if \(G\) is finite-dimensional. However, the theorem has been proven in \cite{MN12} in the case where \(G\) is a Banach--Lie group and \(S\) is an Olshanski semigroup. Olshanski semigroups are semigroups of the form \(\Gamma(C)\) as above with the additional property that the polar map
\[H \times C \rightarrow \Gamma(C), \quad (h,x) \mapsto h\exp(x),\]
is a diffeomorphism.

In \cite{MNO15}, an extension of the L\"uscher--Mack Theorem has been proven for Banach--Lie groups \(G\) and open \(*\)-subsemigroups \(S\) with \(SH = S\). Given a non-degenerate strongly continuous \(*\)-representation \((\pi,\cH)\) of \(S\) with additional smoothness properties, there exists a strongly continuous unitary representation \((\pi^H,\cH)\) of \(H\) and a strongly continuous unitary representation \((\pi^c,\cH)\) of \(G^c\) such that, for \(s \in S,\, h \in H,\) and \(x \in \fq\) satisfying \(\exp(tx) \in S\) for \(t > 0\), we have
\begin{equation}
  \label{eq:intro-ancont}
\pi(sh) = \pi(s)\pi^H(h) \quad \text{and} \quad \pi(\exp(x)) = e^{-i\p\pi^c(ix)}.
\end{equation}

The solution of analytic continuation problems plays an important role in reflection positivity:
In constructive Quantum Field Theory, one uses reflection positivity to construct relativistic field theories from euclidean ones (cf.\ \cite{GJ81,OS73,OS75}).
In the representation theory of Lie groups, one would thus like to pass from a unitary representation of a Lie group \(G\) to a unitary representation of \(G^c\).
The primary example of this passage is from the euclidean motion group \(G = \R^d \rtimes \OO_d(\R)\) to the Poincar\'e group \(G^c = \R^d \rtimes \OO_{1,d-1}(\R)\).
One way to approach this problem involves constructing from the representation of \(G\) a contraction representation of an involutive subsemigroup \(S \subset G\) as shown in \cite[3.4]{NO17}.
If \(S\) has interior points and the assumptions of the L\"uscher--Mack Theorem are satisfied, then we can extend this semigroup representation to a unitary representation of \(G^c\) by analytic continuation.

A priori, it is not always clear whether an implicitly specified subsemigroup \(S\) is an Olshanski semigroup or even satisfies \(S = SH\):
For example, in the context of standard subspaces, i.e.\ real closed subspaces \(V \subset \cH\) such that \(V \cap iV = 0\) and \(\oline{V + iV} = \cH\), the inclusion order on the set of standard subspaces is of particular interest because it relates naturally to inclusions of von Neumann algebras (cf.\ \cite{NO17}). For a unitary representation \((\pi,\cH)\) of \(G\), the semigroup
\[S_V := \{g \in G : U_g V \subset V\} \subset G\]
encodes the order structure on the orbit \(\pi(G).V\) (cf.\ \cite{Ne17}). Furthermore, one can construct from \((\pi,\cH)\) a strongly continuous contraction representation of \(S_V\) on \(\cH\), and its analytic continuation to \(G^c\), if it exists, provides more information about the semigroup \(S_V\).
We will elaborate on this example in Section \ref{sec:examples}.

This article will solve the analytic extension problem for open \(*\)-subsemigroups of finite dimensional Lie groups by refining some of the methods used in \cite{MNO15}.
These extensions are uniquely determined by properties similar to \eqref{eq:intro-ancont}.

We proceed as follows.
In Section \ref{sec:iker}, we recall the concept of positive definite distribution kernels which are invariant under the action of a Lie algebra based on the arguments in \cite{MNO15}.
We then apply Simon's Exponentiation Theorem \cite{Si72} to proof the existence of a unitary representation of the 1-connected Lie group \(G_1^c\) with Lie algebra \(\g_1^c = [\fq,\fq] \oplus i\fq\) (cf.\ Theorem \ref{thm:iker-si-rep}).
One of the main ingredients of our arguments is Fr\"ohlich's Theorem \cite{Fro80}, which gives a criterion for the essential selfadjointness of unbounded operators on Hilbert spaces.
In Section \ref{sec:ancont-rep}, we explain how the methods developed in the previous section can be used in the context of \(*\)-representations of subsemigroups by applying a well-known GNS construction.
We obtain a unitary representation of \(G_1^c\) which we call the \emph{analytic continuation} to \(G_1^c\) (cf.\ Theorem \ref{thm:si-integration}).
We then justify this naming by showing that this representation is an analytic continuation of the original semigroup representation (cf.\ Theorem \ref{thm:semigrp-rep-int}).
In order to extend the analytic continuation to a unitary representation of \(G^c\), we impose that there exists a \(\1\)-neighborhood \(B \subset H := (G^\tau)_0\) such that, for the subsemigroup
\[S_B := \{s \in S : Bs \subset S\},\]
the restriction of the semigroup representation to \(S_B\) is non-degenerate.
Under this condition, we prove the existence of an analytic continuation of the original semigroup representation to the Lie group \(G^c\) (cf.\ Theorem \ref{thm:localrep-analytic-cont}).
In Section \ref{sec:ext-semgrp-rep}, we consider semigroups \(S\) with the property that there exists an open set \(U \subset \fq\) such that, for all \(x \in U\) and \(t > 0\), we have \(\exp(tx) \in S\).
While this property is satisfied for Olshanski semigroups, not all open \(*\)-subsemigroups are of this kind (cf.\ Example \ref{ex:cayley-type}).
We then prove that, for every strongly continuous non-degenerate \(*\)-representation \((\pi,\cH)\) of \(S\), there exists a representation \((\tilde\pi,\cH)\) of an open \(*\)-subsemigroup \(\tilde S\) of the universal covering of \(G\) such that \(\tilde\pi\) is an extension of \(\pi\lvert_{\exp(U)}\) up to coverings (cf.\ Theorem \ref{thm:Ol-semgrp-ext}).
The semigroup \(\tilde S\) is a generalized version of an Olshanski semigroup and we show that the analytic continuations of the representations of \(S\) and \(\tilde S\) to \(G^c\) coincide (cf.\ Corollary \ref{cor:Ol-semgrp-ancont}).
Finally, in Section \ref{sec:examples}, we consider reflection positive representations and symmetric Lie groups with 3-graded Lie algebras as examples for which our results on analytic continuations can be applied. 

\subsection*{Notation and conventions}

For a complex Hilbert space \(\cH\), its scalar product \(\la \cdot, \cdot \ra\) is linear in the second argument. The algebra of bounded operators on \(\cH\) will be denoted by \(B(\cH)\) and the group of unitary operators by \(\U(\cH)\).

For a symmetric Lie group \((G,\tau)\), we set \(g^* := \tau(g)^{-1}\) for \(g \in G\). The corresponding decomposition of \(\g = \L(G)\) into \(\tau\)-eigenspaces is denoted by \(\g = \h \oplus \fq\), where \(\fh\) is the \((+1)\)-eigenspace and \(\fq\) is the \((-1)\)-eigenspace. Moreover, we define \(\g_1 := [\fq,\fq] \oplus \fq\) as the ideal in \(\g\) generated by \(\fq\) and \(\g^c := \h \oplus i\fq\) as the \emph{dual Lie algebra of \(\g\)}.

Let \((\pi,\cH)\) be a strongly continuous unitary representation of \(G\) on a complex Hilbert space \(\cH\) and let \(x \in \g\). Following \cite{Schm90}, we denote the infinitesimal generator of the unitary one-parameter group \(t \mapsto \pi(\exp(tx))\) by \(\partial\pi(x)\) (cf. Appendix \ref{sec:diffvectors}).

Given a smooth manifold \(M\), we denote by \(C_c^\infty(M)\) the space of complex-valued smooth functions on \(M\) and by \(C^{-\infty}(M)\) the space of distributions, i.e.\ antilinear continuous functionals on \(C_c^\infty(M)\).

\section{Invariant positive definite kernels}
\label{sec:iker}

In this section, we recall some of the fundamental properties of positive definite kernels from \cite{MNO15} and explain how invariant positive definite kernels can be used to construct unitary representations.

\begin{definition}
  Let \(X\) be a set. A function \(K: X \times X \rightarrow \C\) is called a \emph{positive definite kernel} if each finite subset \(\{(x_1,\lambda_1),\ldots,(x_n,\lambda_n)\} \subset X \times \C\) satisfies
  \begin{equation}
    \label{eq:posdef-condition}
    \sum_{j,k=1}^n \lambda_j\oline{\lambda_k}K(x_j,x_k) \geq 0.
  \end{equation}
\end{definition}

Every positive definite kernel \(K : X \times X \rightarrow \C\) uniquely determines a Hilbert space \(\cH_K \subset \C^X\) of complex-valued functions on \(X\) for which the point evaluations
\[K_x : \cH_K \rightarrow \C, \quad f \mapsto f(x),\]
are continuous linear functionals. By identifying \(K_x\) with the function in \(\cH_K\) for which \(\la K_x, f\ra = K_x(f)\) for all \(f \in \cH_K\), we obtain
\[K(x,y) = \la K_x, K_y\ra = K_y(x), \quad \text{for } x,y \in X.\]
Furthermore, the space \(\cH_K^0 := \spann \{K_x : x \in X\}\) is dense in \(\cH_K\) (cf.\ \cite[Thm.\ I.1.4]{Ne00}).
The space \(\cH_K\) is called the \emph{reproducing kernel Hilbert space of \(K\)}.

If \(X\) is a topological space and \(K\) is separately continuous and locally bounded, then \(\cH_K \subset C(X)\) (cf.\ \cite[Prop.\ I.1.9]{Ne00}). Similarly, one can show that, if \(X\) is a locally convex smooth manifold and \(K \in C^\infty(X \times X)\), then \(\cH_K \subset C^\infty(X)\).

Consider now a finite dimensional smooth manifold \(M\) and a distribution \(K \in C^{-\infty}(M \times M)\) such that \((\psi_1,\psi_2) \mapsto K(\psi_1 \otimes \oline{\psi_2})\) is a positive semidefinite hermitian form on \(C_c^\infty(M)\) and denote the corresponding Hilbert space completion by \(\cH_K\). The adjoint of the inclusion \(\iota: C_c^\infty(M) \rightarrow \cH_K\) is a continuous injective linear map
\begin{equation}
  \label{eq:reprod-dist-incl}
  \iota' : \cH_K \rightarrow C^{-\infty}(M), \quad \iota'(v)(\varphi) = \la K_\varphi, v\ra,
\end{equation}
(cf.\ \cite[p.\ 47]{MNO15}), so that we can from now on identify \(\cH_K\) with a subspace of \(C^{-\infty}(M)\). The map \(K\) is called a \emph{positive definite distribution}.

\begin{definition}
  Let \(M\) be a finite dimensional smooth manifold. We denote the set of vector fields on \(M\) by \(\cV(M)\).

  (a) Let \(X \in \cV(M)\) and let \(\Phi: \cD \rightarrow M\) be its maximal local flow, where \(\cD \subset \R \times M\) is an open set containing \(\{0\} \times M\). For a smooth function \(f \in C^\infty(M)\) on \(M\), its \emph{Lie derivative} is defined by
  \[\cL_X f = \lim_{t \rightarrow 0} \frac{1}{t} (f \circ \Phi_t - f) \in C^\infty(M).\]
  (b) Let \(M\) be a finite dimensional smooth manifold, \(D \in C^{-\infty}(M \times M)\) be a distribution, and \(\g = \fh \oplus \fq\) be a symmetric Lie algebra with involution \(\tau\). For a vector field \(X \in \cV(M)\), we define:
  \[(\cL_X^1 D)(\psi_1 \otimes \psi_2) := -D(\cL_X \psi_1 \otimes \psi_2) \quad \text{and} \quad (\cL_X^2 D)(\psi_1 \otimes \psi_2) := -D(\psi_1 \otimes \cL_X \psi_2), \quad \psi_1,\psi_2 \in C_c^\infty(M).\]
  Let \(\sigma : \g \rightarrow \cV(M)\) be a homomorphism of Lie algebras.
  Then \(D\) is called \emph{\(\sigma\)-compatible} if \(\cL_{\sigma(x)}^1 D = -\cL_{\sigma(\tau(x))}^2 D\) for all \(x \in \g\).
\end{definition}

For the following proposition, we recall that a vector field \(X \in \cV(M)\) on a smooth manifold \(M\) acts on the space of distributions of \(M\) by
\[(\cL_X D)(\varphi) := -D(\cL_X \varphi), \quad D \in C^{-\infty}(M),\varphi \in C^{\infty}(M).\]

\begin{prop}
  \label{prop:iker-geom-fro}
  Let \(M\) be a finite dimensional smooth manifold and \(K \in C^{-\infty}(M \times M)\) be a positive definite distribution. Let \(\g = \fh \oplus \fq\) be a symmetric Lie algebra with involution \(\tau\) and \(\sigma : \g \rightarrow \cV(M)\) be a homomorphism of Lie algebras such that \(K\) is \(\sigma\)-compatible. For \(x \in \g\), let
  \[\cD_{x} := \{D \in \cH_K : \cL_{\sigma(x)}D \in \cH_K\}.\]
  and define
  \[\cL^K_{x} : \cD_{x} \rightarrow \cH_K, \quad D \mapsto \cL_x^K D := \cL_{\sigma(x)} D.\]
  Then the following assertions hold:
  \begin{enumerate}
    \item For all \(x \in \g\), we have \(\cH_K^0 := \spann \{K_\varphi : \varphi \in C_c^{\infty}(M)\} \subset \cD_{x}\) and \(\cL^K_{x}K_{\varphi} = K_{\cL_{\sigma(\tau(x))}\varphi} \in \cH_K^0\) for \(\varphi \in C_c^\infty(M)\). Moreover, \(\cL_x^K\) is closed.
    \item For \(x \in \fh\), the operator \(\cL^K_{x}\lvert_{\cH_K^0}\) is skew-symmetric with \((\cL_x^K\lvert_{\cH_K^0})^*= -\cL_x^K\).
    \item For \(y \in \fq\), let \(\Phi^{\sigma(y)}\) be the maximal local flow of \(\sigma(y)\). Then \(\cL^K_{y}\) is selfadjoint and \(\cH_K^0\) is a core of \(\cL^K_{y}\). Moreover, if \(\varphi \in C_c^\infty(M)\) and \(\Phi^{\sigma(y)}\) is defined on \([-\varepsilon,\varepsilon] \times \supp(\varphi)\) for \(\varepsilon > 0\), then
      \begin{equation}
        \label{eq:iker-geom-fro}
        e^{t\cL^K_{y}}K_\varphi = K_{\varphi \circ \Phi_{-t}^{\sigma(y)}}, \quad |t| \leq \varepsilon.
      \end{equation}
      The curve \(t \mapsto e^{t\cL^K_y}K_\varphi\) has an analytic extension to the strip \(\cS_\varepsilon = \{z \in \C : |\Re z| < \varepsilon\}\). In particular, the space \(\cH_K^0\) consists of analytic vectors of \(i\cL_y^K\).
  \end{enumerate}
\end{prop}
\begin{proof}
  (a) Let \(x \in \g\) and \(\psi_1,\psi_2 \in C_c^\infty(M)\). The \(\sigma\)-compatibility of \(K\) implies that
  \[(\cL_{\sigma(x)}K_{\psi_2})(\psi_1) = -K_{\psi_2}(\cL_{\sigma(x)}\psi_1) = -K(\cL_{\sigma(x)}\psi_1 \otimes \oline{\psi_2}) = K(\psi_1 \otimes \oline{\cL_{\sigma(\tau(x))}\psi_2}) = K_{\cL_{\sigma(\tau(x))}\psi_2}(\psi_1).\]
  To see that \(\cL_x^K\) is closed, it suffices to note that the Lie derivative \(\cL_{\sigma(x)} : C_c^\infty(M) \rightarrow C_c^\infty(M)\) is a continuous linear map on the locally convex space \(C_c^\infty(M)\). Therefore, its adjoint map on \(C^{-\infty}(M)\) is continuous with respect to the weak-*-topology on \(C^{-\infty}(M)\). Since the inclusion map \eqref{eq:reprod-dist-incl} is continuous, the closedness follows from the definition of \(\cD_x\)\footnote{More generally, consider a topological vector space \(E\) and a Hilbert space \(\cH\) such that \(\cH\) is a subspace of \(E\) and the inclusion \(\cH \hookrightarrow E\) is continuous. Then, for any continuous linear map \(T : E \rightarrow E\), the operator \(T_\cH : \cD \rightarrow \cH\) defined on \(\cD := T^{-1}(\cH) \cap \cH\) by \(T_\cH(v) := T(v)\) is a closed operator. }.
  
  (b) Let \(x \in \fh\). The computations in the proof of (a) show that \(\cL_x^K\lvert_{\cH_K^0}\) is skew-symmetric. We now prove that \(\cL_x^K = -(\cL_x^K\lvert_{\cH_K^0})^*\): Let \(D \in \cD_x\). Then we have for all \(\varphi \in C_c^\infty(M)\):
  \[- \la K_\varphi, \cL_x^K D\ra = -(\cL_x^K D)(\varphi) = D(\cL_{\sigma(x)}\varphi) = \la K_{\cL_{\sigma(x)}\varphi}, D\ra = \la \cL_x^K K_\varphi, D\ra.\]
  Thus, \(D \in \cD((\cL_x^K\lvert_{\cH_K^0})^*)\) and \((\cL_x^K\lvert_{\cH_K^0})^*D = -\cL_x^K D\). Conversely, let \(D \in \cD((\cL_x^K\lvert_{\cH_K^0})^*)\) with \(E := (\cL_x^K\lvert_{\cH_K^0})^*D\). Then
  \[E(\varphi) = \la K_\varphi, E\ra = \la \cL_x^K K_\varphi, D\ra = \la K_{\cL_{\sigma(x)}\varphi}, D\ra = D(\cL_{\sigma(x)}\varphi) = -(\cL_{\sigma(x)}D)(\varphi)\]
  for all \(\varphi \in C_c^\infty(M)\) shows that \(\cL_{\sigma(x)}D = -E \in \cH_K\) and thus \(D \in \cD_x\) with \(\cL_x^KD = -E\). This proves the claim.

  In order to show (c), we note that, by the Geometric Fr\"ohlich Theorem for distributions \cite[Thm.\ 7.5]{MNO15}, the operator \(\cL_y^K\lvert_{\cH_K^0}\) is essentially selfadjoint, its closure equals \(\cL_y^K\), and \eqref{eq:iker-geom-fro} holds. By the spectral theorem, the curve
  \[\cS_\varepsilon \rightarrow \cH_K, \quad u + iv \mapsto e^{u\cL_y^K}e^{iv\cL_y^K}K_\varphi,\]
  is continuous on \(\cS_\varepsilon\) and analytic on \(\cS_\varepsilon \setminus i\R\). Thus, it is analytic on \(\cS_\varepsilon\). This proves the second part of (c).
\end{proof}

\begin{thm}
  \label{thm:iker-si-rep}
  In the context of {\rm Proposition \ref{prop:iker-geom-fro}}, set for \(x \in [\fq,\fq]\) and \(iy \in i\fq \subset \g_\C\)
  \[T(x) := \cL^K_{x}\lvert_{\cH_K^0}, \quad T(iy) := i\cL^K_{y}\lvert_{\cH_K^0}.\]
  Then \(T\) defines a representation of \(\g_1^c = [\fq,\fq] \oplus i\fq\) by unbounded skew-symmetric operators. Furthermore, there exists a unique strongly continuous unitary representation \((\pi_1^c,\cH_K)\) of the 1-connected Lie group \(G_1^c\) with Lie algebra \(\g_1^c\) such that 
  \[\partial\pi_1^c(x)\lvert_{\cH_K^0} = \cL^K_x\lvert_{\cH_K^0} \quad \text{and} \quad \partial\pi_1^c(iy) = i\cL^K_y \quad \text{for all } x \in [\fq,\fq], y \in \fq.\]
\end{thm}
\begin{proof}
  By Proposition \ref{prop:iker-geom-fro}, the space \(\cH_K^0\) is invariant under \(T\) and consists of analytic vectors of \(T(iy)\) for all \(y \in \fq\). Since \(i\fq\) generates \(\g_1^c\), Simon's Exponentiation Theorem \cite[Cor.\ 2]{Si72} implies the existence of a strongly continuous unitary representation \((\pi_1^c, \cH)\) of \(G_1^c\) which is uniquely determined by \(\p\pi_1^c(x)\lvert_{\cH_K^0} = T(x)\) for all \(x \in \g_1^c\). For \(y \in \fq\), the operator \(\p\pi_1^c(iy)\lvert_{\cH_K^0}\)  is essentially skew-adjoint by Nelson's Theorem (cf.\ \cite[Thm.\ X.39]{ReSi75}), and its closure coincides with \(\p\pi_1^c(iy)\). Combining this with Proposition \ref{prop:iker-geom-fro}(c), we obtain
    \[i\cL_{y}^K = \oline{i\cL_{y}^K\lvert_{\cH_K^0}} = \oline{\p\pi_1^c(iy)\lvert_{\cH_K^0}} = \p\pi_1^c(iy).\qedhere\]
\end{proof}

\begin{rem}
  \label{rem:si-integration}
  Let \(K\) be as in Proposition \ref{prop:iker-geom-fro} and let \((\pi^c_1,\cH_K)\) be the representation of \(G_1^c\) we obtain from Theorem \ref{thm:iker-si-rep}.
  Then, for \(y \in \fq\), we have
  \[-\p \pi_1^c(iy) = (\p \pi_1^c(iy)\lvert_{\cH_K^0})^*= (i\cL^K_{y}\lvert_{\cH_K^0})^* = -i\oline{\cL^K_{y}\lvert_{\cH_K^0}} = -i\cL^K_y.\]
  For \(x \in [\fq,\fq] \subset \fh\), the operator \(\p \pi_1^c(x)\) is a skew-adjoint extension of \(\cL^K_x\lvert_{\cH_K^0}\). Hence, we have by Proposition \ref{prop:iker-geom-fro}(b)
  \[\cL^K_x\lvert_{\cH_K^0} \subset \p \pi_1^c(x) \subset (-\cL^K_x\lvert_{\cH_K^0})^* = \cL_x^K.\]
\end{rem}

\section{Analytic continuation of \(*\)-semigroup representations}
\label{sec:ancont-rep}

We now apply the results from Section \ref{sec:iker} to representations of \(*\)-subsemigroups in order to construct unitary representations of Lie groups. Throughout this section, \((G,\tau)\) denotes a symmetric Lie group with Lie algebra \(\g = \fh \oplus \fq\), \(S \subset G\) is an open \(*\)-subsemigroup, and \(\cH\) is a complex Hilbert space. Furthermore, we fix a right-invariant Haar measure on \(G\).

\subsection{Semigroup representations and invariant distribution kernels}

A function \(\varphi: S \rightarrow \C\) is called \emph{positive definite} if \(K^\varphi(s,t) := \varphi(st^*)\) is a positive definite kernel.

\begin{prop}
  \label{prop:pd-fun-dist}
  Let \(\varphi: S \rightarrow \C\) be a continuous positive definite function and define
  \begin{equation}
    \label{eq:pd-fun-dist}
    K(\psi_1 \otimes \oline{\psi_2}) := \la \psi_1, \psi_2 \ra_\varphi := \int_S\int_S \oline{\psi_1(g)}\psi_2(h)\varphi(gh^*)\,dgdh, \quad \psi_1,\psi_2 \in C_c^\infty(S).
  \end{equation}
  Let \(\sigma : \g \rightarrow \cV(S),\, \sigma(x)(s) := \derat{0}s\exp(tx),\) be the usual homomorphism from \(\g\) onto the Lie algebra of left-invariant vector fields on \(G\) restricted to \(S\). Then \(K\) is a positive definite \(\sigma\)-compatible distribution.
\end{prop}
\begin{proof}
  The positivity of \(K\) follows from the positive definiteness of the kernel \(K^\varphi\) and the fact that we can approximate the integral in \eqref{eq:pd-fun-dist} by sums in the form of \eqref{eq:posdef-condition}.
  For \(x \in \g\), the flow of \(\sigma(x)\) is given by
  \[\Phi^x : \R \times G \rightarrow G, \quad \Phi^x_t(g) := g\exp(tx).\]
  Hence we have for \(\psi_1,\psi_2 \in C_c^\infty(S)\):
  \begin{align*}
    \la \cL_{\sigma(x)} \psi_1, \psi_2 \ra_\varphi &= \int_G\int_G \oline{(\cL_{\sigma(x)}\psi_1)(g)}\psi_2(h)\varphi(gh^*)\, dg dh\\
                                         &= \int_G\int_G \derat{0} \oline{\psi_1(g\exp(tx))}\psi_2(h) \varphi(gh^*)\, dg dh \\
                                         &= \derat{0} \int_G\int_G \oline{\psi_1(g\exp(tx))}\psi_2(h) \varphi(gh^*)\, dg dh \\
                                         &= \derat{0} \int_G\int_G \oline{\psi_1(g)}\psi_2(h)\varphi(g\exp(-tx)h^*)\, dg dh \\
                                         &= \derat{0} \int_G\int_G \oline{\psi_1(g)}\psi_2(h\tau(\exp(-tx)))\varphi(gh^*)\, dg dh \\
                                         &= \int_G \int_G \oline{\psi_1(g)}(\cL_{-\sigma(\L(\tau)(x))}\psi_2)(h) \varphi(gh^*) \, dg dh \\
                                         &= \la \psi_1, \cL_{-\sigma(\L(\tau)(x))}\psi_2 \ra_\varphi. \qedhere
  \end{align*}
  This shows that \(K\) is \(\sigma\)-compatible.
\end{proof}

For every strongly continuous representation \((\pi,\cH)\) of \(S\) and \(v \in \cH\), the matrix coefficient \(\pi^{v,v}(s) := \la v, \pi(s)v\ra\) is a continuous positive definite function. If \((\pi,\cH,v)\) is cyclic, i.e.\ \(v \in \cH\) is such that \(\pi(S)v\) generates a dense subspace in \(\cH\), then the following proposition shows that \((\pi,\cH)\) is unitarily equivalent to a representation on a space of distributions.

\begin{prop}
  \label{prop:cycl-rep-distkernel}
  Let \((\pi,\cH,v)\) be a strongly continuous cyclic \(*\)-representation of \(S\) on \(\cH\). Let \(K \in C^{-\infty}(S \times S)\) be defined as in \eqref{eq:pd-fun-dist}, where \(\varphi = \pi^{v,v}\). Then \((\pi,\cH)\) is unitarily equivalent to a \(*\)-representation \((\pi_K, \cH_K)\) of \(S\) on \(\cH_K\) with the following property:
  For every function \(\psi \in C_c^\infty(S)\) and \(x \in \fq\), there exists \(\varepsilon > 0\) such that
  \[\pi_K(s\exp(tx))K_\psi = \pi_K(s)K_{\psi \circ \Phi_{-t}^x}\]
  for all \(s \in S\) and \(|t| < \varepsilon\) with \(s\exp(tx) \in S\).
\end{prop}
\begin{proof}
  In the following, we define \(\pi^{u,w}(s) := \la u, \pi(s)w \ra\) for \(u,w \in \cH, s \in S\).
  Consider the map
  \[\gamma : C_c^\infty(S) \rightarrow \cH, \quad \gamma(\psi) := \int_S \psi(s)\pi(s^*)v\, ds.\]
  The range of \(\gamma\) is dense in \(\cH\) because if \(w \in (\im \gamma)^\bot\), then we have for all \(\psi \in C_c^\infty(S)\)
  \[0 = \la \gamma(\psi), w \ra = \int_S \oline{\psi(s)} \la \pi(s^*)v, w \ra \,ds,\]
  which implies \(\pi^{v,w}(s) = 0\) for all \(s \in S\). Since \(v\) is cyclic, this implies \(w = 0\).
  Furthermore, we have
  \begin{align*}
    \la \gamma(\psi_1), \gamma(\psi_2) \ra &= \int_S\int_S \oline{\psi_1(s)}\psi_2(t) \la \pi(s^*) v, \pi(t^*)v \ra \, dsdt = \int_S\int_S \oline{\psi_1(s)}\psi_2(t) \pi^{v,v}(st^*)\,dsdt\\
                                           &= K(\psi_1 \otimes \oline{\psi_2}).
  \end{align*}
  Thus, the Realization Theorem for positive definite kernels \cite[Thm.\ I.1.6]{Ne00} implies that
  \begin{equation}
    \label{eq:rkhs-realization-isom}
    \Psi : \cH \rightarrow \cH_K, \quad \Psi(w)(\psi) := \la \gamma(\psi), w \ra
  \end{equation}
  is a unitary operator with \(\Psi(\gamma(\psi)) = K_\psi\) for \(\psi \in C_c^\infty(S)\). We obtain a strongly continuous \(*\)-representation of \(S\) on \(\cH_K\) by setting \(\pi_K(s) := \Psi \circ \pi(s) \circ \Psi^*\).

  It remains to show the second part of the claim: Let \(\psi \in C_c^\infty(S)\) and \(x \in \fq\). We choose \(\varepsilon > 0\) such that \(g\exp(tx) \in S\) for all \(g \in \supp(\psi)\) and \(|t| < \varepsilon\). Let \(s \in S\), \(|t| < \varepsilon\) such that \(s\exp(tx) \in S\). Then we have for all \(f \in C_c^\infty(S)\):
  \begin{align*}
    (\pi_K(s\exp(tx))K_\psi)(f) &= (\pi_K(s\exp(tx))\Psi(\gamma(\psi)))(f) = \la \gamma(f), \pi(s\exp(tx)) \gamma(\psi) \ra \\
                                &= \int_G\int_G \oline{f(g)}\psi(h) \la \pi(g^*)v, \pi(s\exp(tx)h^*)v \ra \, dgdh \\
                                &= \int_G\int_G \oline{f(g)}\psi(h) \la \pi(g^*)v, \pi(s(h\exp(tx))^*)v \ra \, dgdh \\
                                &= \int_G\int_G \oline{f(g)}\psi(h\exp(-tx)) \la \pi(g^*)v, \pi(sh^*)v \ra \, dgdh \\
                                &= \int_G\int_G \oline{f(g)}(\psi \circ \Phi_{-t}^x)(h) \la \pi(g^*)v, \pi(sh^*)v \ra \, dgdh \\
                                &= (\pi_K(s)K_{\psi \circ \Phi^x_{-t}})(f) \qedhere
  \end{align*}
\end{proof}

\begin{coro}
  \label{cor:rkhs-identification}
 Let \((\pi,\cH,v)\) be a strongly continuous cyclic \(*\)-representation of \(S\) on \(\cH\).
 For a continuous function \(f \in C(S)\), we denote by \(D_f\) the distribution
  \[D_f : C_c^\infty(S) \rightarrow \C, \quad \psi \mapsto \int_S \oline{\psi(s)}f(s)\, ds.\]
  Then, with the notation from {\rm Proposition \ref{prop:cycl-rep-distkernel}}, we obtain a unitary operator
  \[\Psi: \cH \rightarrow \cH_K, \quad w \mapsto D_{\pi^{v,w}}.\]
\end{coro}

\begin{rem}
  \label{rem:rkhs-identification}
  Let \((\pi,\cH,v_0)\) be a strongly continuous cyclic \(*\)-representation of \(S\) on \(\cH\) and let \(K\) be defined as in \eqref{eq:pd-fun-dist} with \(\varphi = \pi^{v_0,v_0}\). We identify every \(x \in \g\) with the vector field \(\sigma(x)(s) := \derat{0}s\exp(tx)\) on \(S\). Then we can use the unitary operator \eqref{eq:rkhs-realization-isom} from Proposition \ref{prop:cycl-rep-distkernel} to identify the operators \(\cL^K_x\) from Proposition \ref{prop:iker-geom-fro} on \(\cH_K\) with operators on \(\cH\): Therefore, \(\cH\) contains a dense subspace
  \[\cH^0 := \spann \{\pi(f)v_0 : f \in C_c^\infty(S)\}, \quad \text{where } \pi(f) := \int_S f(g)\pi(g^*)\, dg.\]
  For every \(x \in \fg\), there exists a densely defined operator
  \[\cL_x^\pi : \cD(\cL_x^\pi) \rightarrow \cH, \quad \cD(\cL_x^\pi) := \{w \in \cH : (\exists v \in \cH)\, \cL_x^K D_{\pi^{v_0,w}} = D_{\pi^{v_0,v}}\}\]
  (cf.\ Corollary \ref{cor:rkhs-identification}) with \(\cH^0 \subset \cD(\cL_x^\pi)\) and
  \begin{equation}
    \label{eq:rkhs-id-liederiv}
    \cL_x^\pi \pi(f)v_0 = \pi(\cL_{\L(\tau)(x)}f)v_0, \quad \text{for } f \in C_c^\infty(S).
  \end{equation}
  If \(x \in \fq\), then \(\cL_x^\pi\) is selfadjoint and \(\cH^0\) is a core of \(\cL_x^\pi\) (cf.\ Proposition \ref{prop:iker-geom-fro}(c)).
\end{rem}
By applying Theorem \ref{thm:iker-si-rep} to the case of positive definite distributions induced by \(*\)-subsemigroup representations, we obtain the following
\begin{thm}
  \label{thm:si-integration}
  Let \((\pi,\cH)\) be a strongly continuous non-degenerate \(*\)-representation of \(S\). Then there exists a unique strongly continuous unitary representation \((\pi_1^c,\cH)\) of the 1-connected Lie group \(G_1^c\) with Lie algebra \(\g_1^c = [\fq,\fq] \oplus i\fq\) such that, for each \(x \oplus iy \in [\fq,\fq] \oplus i\fq\), the infinitesimal generator \(\p\pi_1^c(x+iy)\) of the one-parameter group \(t \mapsto \pi_1^c(\exp(t(x+iy)))\) satisfies
  \begin{equation}
    \label{eq:si-integration}
    \p\pi_1^c(x+iy)\pi(f) = \pi((\cL_x-i\cL_y)f) \quad \text{for all } f \in C_c^\infty(S).
  \end{equation}
\end{thm}
\begin{proof}
  Since \(\pi\) is non-degenerate, there exists a decomposition \((\pi,\cH) \cong \hat\bigoplus_{j \in J} (\pi_j, \cH_j, v_j)\) of \(\pi\) into cyclic subrepresentations.
  For each \(j \in J\), let \(K^j\) be the positive definite distribution we defined in Proposition \ref{prop:cycl-rep-distkernel}.
  Then we obtain a continuous unitary representation \((\pi_1^c,\cH)\) on \(G_1^c\) by applying Theorem \ref{thm:iker-si-rep} and Remark \ref{rem:rkhs-identification} to each \(K^j\).
  Let now \(f \in C_c^\infty(S)\) and \(x \oplus iy \in [\fq,\fq] \oplus i\fq\).
  Since \(\p\pi_1^c(x + iy)\) is closed and \(\pi(f)\) is a continuous operator, it suffices to show \eqref{eq:si-integration} on a dense subspace.
  But on each of the subspaces \(\pi_j(S)v_j\), equation \eqref{eq:si-integration} follows from \eqref{eq:rkhs-id-liederiv}.
  Hence, it also holds on \(\cH\).
  The uniqueness of \(\pi_1^c\) follows from the uniqueness on the subspaces \(\cH_j\) for \(j \in J\) (cf.\ Theorem \ref{thm:iker-si-rep}).
\end{proof}

We call \((\pi_1^c,\cH)\) the \emph{analytic continuation of \((\pi,\cH)\) to \(G_1^c\)}.

\begin{rem}
  \label{rem:rkhs-ident-noncyclic}
  By construction, the infinitesimal generators of the analytic continuation \((\pi_1^c,\cH)\) to \(G_1^c\) are direct sums of the Lie derivation operators we constructed in Remark \ref{rem:rkhs-identification}: Let \((\pi,\cH) = \hat\bigoplus_{j \in I} (\pi_j,\cH_j,v_j)\) be a decomposition of \(\pi\) into cyclic subrepresentations. For \(x \in \g\), we consider the operator
  \[\cL_x^\pi : \cD(\cL_x^\pi) \rightarrow \cH, \quad \cD(\cL_x^\pi) := \{(v_j)_{j \in I} \in \hat\bigoplus_{j \in I} \cH_j : (\forall j \in I) v_j \in \cD(\cL_x^{\pi_j}), \, \sum_{j \in I} \|\cL_x^{\pi_j}v_j\|^2 < \infty\},\]
  with \(\cL_x^\pi(v_j)_{j \in I} := (\cL_x^{\pi_j}v_j)_{j \in I}\). In particular, for \(x \in [\fq,\fq]\) and \(y \in \fq\), the operator \(\cL_{x + y}^\pi\) is defined on the dense subspace \(\cH^0 := \spann\{\pi(f)\cH : f \in C_c^\infty(S)\}\) and we have
  \[i\cL_y^\pi = \partial\pi_1^c(iy) \quad \text{and} \quad \cL_x^\pi\lvert_{\cH^0} = \partial\pi_1^c(x)\lvert_{\cH^0}\]
  because \(\cH^0\) is a core of \(\cL_y^\pi\) (cf. Proposition \ref{prop:iker-geom-fro}(c)).
\end{rem}

\subsection{The analytic continuation}

Up to this point, we have only shown that a strongly continuous unitary representation of \(G_1^c\) can be constructed from a strongly continuous \(*\)-representation of \(S\). In this section, we will explain how these two representations are related and, in particular, why the name ``analytic continuation'' is justified.

\begin{lem}
  \label{lem:dom-exp-fro}
  Let \(H : \cD(H) \rightarrow \cH\) be a (possibly unbounded) selfadjoint operator on \(\cH\) and let \(t_-,t_+ \in \R\) such that \(t_- \leq 0 < t_+\). Then, for every \(v \in \cH\), the following statements are equivalent:
  \begin{enumerate}
    \item \(v \in \cD(e^{t_- H}) \cap \cD(e^{t_+ H})\)
    \item There exists a continuous curve \(\gamma: [t_-,t_+] \rightarrow \cH\) which is differentiable on \((t_-,t_+)\) and solves the initial value problem
      \begin{equation}
        \label{eq:fro-ivp}
        \gamma'(s) = H\gamma(s),\quad \gamma(0) = v \quad \text{for } s \in (t_-,t_+).
      \end{equation}
  \end{enumerate}
  If the above conditions are satisfied for a vector \(v \in \cH\), then the unique solution of \eqref{eq:fro-ivp} is given by \(\gamma(t) = e^{tH}v\) for \(t \in [t_-,t_+]\).
\end{lem}
\begin{proof}
  Let \(P\) be the spectral measure corresponding to \(H\) and set \(P^w(E) := \la P(E)w,w\ra\), where \(w \in \cH\) and \(E \subset \R\) is Borel-measurable. Then we have
  \[\la w,Hw \ra = \int_{-\infty}^\infty \lambda \,dP^w (\lambda)\]
  for \(w \in \cD(H)\) and, for a Borel-measurable function \(f\) on \(\R\), we have \(w \in \cD(f(H))\) if and only if \(f \in L^2(\R, P^w)\).

  Suppose that \(v \in \cD(e^{t_- H}) \cap \cD(e^{t_+ H})\).
  Then, using the spectral integral representation of \nolinebreak \(H\), we see that \(v \in \cD(e^{t H})\) for \(t \in [t_-,t_+]\) and \(v \in \cD(He^{s H})\) for \(s \in (t_-,t_+)\).
  Hence, we can define the curve \(\gamma(t) := e^{t H}v\), \(t \in [t_-,t_+]\), which is continuous on \([t_-,t_+]\) and differentiable on \((t_-,t_+)\) by spectral calculus with \(\gamma'(s) = H\gamma(s)\) for \(s \in (t_-,t_+)\).

  Conversely, let \(\gamma: [t_-,t_+] \rightarrow \cH\) be a solution of \eqref{eq:fro-ivp}. We apply the following argument from the proof of \cite[Thm.\ I.1]{Fro80} in order to prove (a):
  For \(m > 0\), let \(E_m := \chi_{[-m,m]}(H)\) be the spectral projection corresponding to the interval \([-m,m]\), and define
  \[\gamma_m(t) := E_m\gamma(t), \quad H_m := E_m H = H E_m, \quad t \in [t_-,t_+].\]
  Since \(H_m\) is a bounded operator and \(\gamma_m\) satisfies
  \[\gamma_m'(s) = E_m\gamma'(s) = E_m H\gamma(s) = H_m\gamma_m(s), \quad s \in (t_-,t_+),\]
  we obtain
  \[\gamma_m(t) = e^{tH_m}\gamma_m(0) = e^{t H} \gamma_m(0), \quad t \in [t_-,t_+].\]
  By taking the limit \(m \rightarrow \infty\), we see that
  \[\lim_{m \rightarrow \infty}\gamma_m(0) = \gamma(0) \quad \text{and} \quad \lim_{m \rightarrow \infty} e^{t H} \gamma_m(0) = \lim_{m \rightarrow \infty} \gamma_m(t) = \gamma(t), \quad t \in [t_-,t_+],\]
  which implies \(\gamma(0) = v \in \cD(e^{tH})\) and \(e^{tH}v = \gamma(t)\) for \(t \in [t_-,t_+]\) because \(e^{tH}\) is closed.
\end{proof}

\begin{lem}
  \label{lem:semigrp-smooth-vec}
  Let \((\pi,\cH)\) be a strongly continuous representation of \(S\) and let \(f \in C_c^\infty(S)\). Then the range of the operator 
  \[\pi(f) := \int_S f(s)\pi(s^*)\, ds\]
  on \(\cH\) consists of smooth vectors of \((\pi,\cH)\) in the sense that the orbit map
  \[\pi^v : S \rightarrow \cH, \quad \pi^v(s) = \pi(s)v\]
  is smooth for all \(v \in \im(\pi(f))\).
\end{lem}
\begin{proof}
  Let \(f \in C_c^\infty(S), v \in \cH,\) and let \(T_\pi\) be the \(\cH\)-valued distribution on \(G\) defined by
  \[T_\pi(h) := \int_G \oline{h(g)}(\1_{S}(g)\pi(g^*))v\,dg = \int_S \oline{h(g)}\pi(g^*)v \,dg, \quad h \in C_c^\infty(G).\]
  Then, by using the right-invariance of \(dg\) and regarding \(f\) as an element in \(C_c^\infty(G)\), we see that
  \[s \mapsto \pi^{\pi(f)v}(s) = \int_S f(g)\pi(sg^*)v_0\,dg = \int_S f(g)\pi((gs^*)^*)v_0\,dg =\int_S f(g\tau(s))\,dT_\pi(g)\]
  is smooth by \cite[Prop.\ A 2.4.1]{Wa72}.
\end{proof}

\begin{lem}
  \label{lem:deriv-op-flow}
  Let \((\pi,\cH)\) be a non-degenerate \(*\)-representation of \(S\) and, for \(x \in \g\), let \(\cL^\pi_x\) be defined as in {\rm Remark \ref{rem:rkhs-ident-noncyclic}}. Let \(\Phi^x_t(g) := g\exp(tx), g \in G\), be the flow of the left-invariant vector field corresponding to \(x\). Then
  \[\derat{0}\pi(f \circ \Phi^x_t) = \pi(\cL_x f) = \cL_{\L(\tau)(x)}^\pi(f) \quad \text{for all } f \in C_c^\infty(S).\]
\end{lem}
\begin{proof}
  Let \(v \in \cH\),\(f \in C_c^\infty(S)\), and \(x \in \g\). Then, since the support of \(f\) is compact, we obtain
  \begin{align*}
    \derat{0} \pi(f \circ \Phi_t^x)v &= \derat{0} \int_G f(g\exp(tx))\pi(g^*)v \,dg = \int_G \derat{0} f(g\exp(tx))\pi(g^*)v \,dg\\
                                     &= \int_G \derat{0} (\cL_x f)(g)\pi(g^*) = \pi(\cL_x f)v = \cL_{\L(\tau)(x)}^\pi\pi(f)v.
  \end{align*}
  where the last equality follows from Proposition \ref{prop:iker-geom-fro}(a). 
\end{proof}

\begin{prop}
  \label{prop:semgrp-rep-curve}
  Let \((\pi,\cH)\) be a strongly continuous non-degenerate \(*\)-representation of \(S\) and let \(s \in S\) and \(v \in \cH\). For \(x \in \g\) and \(\varepsilon > 0\) such that \(\exp(tx)s \in S\) for \(|t| < \varepsilon\), consider the curve
  \[\gamma: (-\varepsilon,\varepsilon) \rightarrow \cH, \quad \gamma(t) := \pi(\exp(tx)s)v.\]
  Then the following assertions hold:
  \begin{enumerate}
    \item Let \(\cL_x^\pi\) and \(\cH^0\) be defined as in {\rm Remark \ref{rem:rkhs-ident-noncyclic}}. If \(\gamma\) is differentiable, then \(\gamma(t) \in \cD((\cL_{\L(\tau)(x)}^\pi\lvert_{\cH^0})^*)\) and \((\cL^\pi_{-\L(\tau)(x)}\lvert_{\cH^0})^*\gamma(t) = \gamma'(t)\) for all \(|t| < \varepsilon\).
    \item If \(x \in \fq\), then the following holds:
      \begin{enumerate}
        \item \(\gamma\) is analytic in \((-\varepsilon,\varepsilon)\).
        \item \(\gamma(t) \in \cD(\cL_x^\pi)\) for all \(t \in (-\varepsilon,\varepsilon)\)
        \item \(\gamma\) solves the initial value problem \eqref{eq:fro-ivp} for \(H = \cL_x^\pi\) and the initial value \(\pi(s)v\).
        \item \(\pi(s)v \in \cD(e^{t\cL_x^\pi})\) and \(\gamma(t) = e^{t\cL_x^\pi}\pi(s)v\) for \(t \in (-\varepsilon,\varepsilon)\).
      \end{enumerate}
  \end{enumerate}
\end{prop}
\begin{proof}
  (a) Let \(f \in C_c^\infty(S)\) and \(t \in (-\varepsilon,\varepsilon)\). Recall that the flow \(\Phi^x\) of the left-invariant vector field corresponding to \(x\) is given by \(\Phi^x_h(g) := g\exp(hx)\). There exists \(\delta > 0\) such that \((t-\delta,t+\delta) \subset (-\varepsilon,\varepsilon)\) and \(\Phi_h^x(\supp f) \subset S\) for all \(h \in (-\delta,\delta)\).
  Then, using the right invariance of the Haar measure, we see that for all such \(h\) and \(w \in \cH\)
  \begin{align*}
    \la \pi(f)w, \gamma(t+h) \ra &= \int_S \oline{f(g)} \la \pi(g^*)w, \pi(\exp((t+h)x)s)v \ra \, dg\\
                             &= \int_S \oline{f(g)} \la \pi(s^*\exp(tx)^*(g\exp(hx))^*)w, v \ra \, dg\\
                             &= \int_S \oline{f(g\exp(-hx))} \la \pi(s^*\exp(tx)^*g^*)w, v \ra \, dg\\
                             &= \int_S \oline{(f \circ \Phi_{-h}^{x})}(g) \la \pi(g^*)w, \pi(\exp(tx)s)v \ra \,dg \\
                             &= \la \pi(f \circ \Phi_{-h}^{x})w, \gamma(t) \ra.
  \end{align*}
  By Lemma \ref{lem:deriv-op-flow}, we have \(\derat{0} \pi(f \circ \Phi_h^{x})w = \cL^\pi_{\L(\tau)(x)} \pi(f)w\). Hence, we obtain
  \begin{align*}
    \la \pi(f)w, \gamma'(t) \ra &= \lim_{h \rightarrow 0} \la \tfrac{1}{h} (\pi(f \circ \Phi_{-h}^{x})w - \pi(f)w), \gamma(t) \ra  = \la \cL_{-\L(\tau)(x)} \pi(f)w, \gamma(t) \ra.\\
  \end{align*}
  Since this holds for all \(f \in C_c^\infty(S)\), we conclude that \[\gamma(t) \in \cD((\cL^\pi_{\L(\tau)(x)}\lvert_{\cH^0})^*) \quad \text{and} \quad \gamma'(t) = (\cL^\pi_{-\L(\tau)(x)}\lvert_{\cH^0})^* \gamma(t)\] for \(t \in (-\varepsilon,\varepsilon)\).

  (b) (i) Let now \(x \in \fq\). Note that \(s^*\exp(tx)s \in S\) for \(|t| < 2\varepsilon\) and consider the continuous function
  \[\eta: (-2\varepsilon,2\varepsilon) \rightarrow \C, \quad \eta(t) := \la v, \pi(s^*\exp(tx)s)v \ra.\]
  Then the kernel on \((-2\varepsilon,2\varepsilon) \times (-2\varepsilon,2\varepsilon)\) given by
  \[K^{\eta}(t,t') := \eta(\tfrac{t+t'}{2}) = \la \gamma(\tfrac t 2), \gamma(\tfrac{t'}{2}) \ra\]
  is positive definite. By \cite{Wi34}, there exists a positive Borel measure \(\mu\) on \(\R\) such that
  \[\eta(t) = \cL(\mu)(t) := \int_\R e^{-xt} \,d\mu(x), \quad t \in (-2\varepsilon,2\varepsilon)\]
  and \(\eta\) is analytic on \((-2\varepsilon,2\varepsilon)\). This shows that the kernel \(K^{\eta}\) is analytic. Hence \(\gamma\) is analytic as well by \cite[Thm.\ 5.1]{Ne10b}).

  By Remark \ref{rem:rkhs-ident-noncyclic}, \(\cH^0\) is a core of \(\cL^\pi_x\), so that \(\cL^\pi_x = \oline{\cL^\pi_x\lvert_{\cH^0}} = (\cL^\pi_x\lvert_{\cH^0})^*\). Thus, (a) implies (ii) and (iii).
  
  Now (iv) follows from Lemma \ref{lem:dom-exp-fro}.
\end{proof}

\begin{prop}
  \label{prop:semigrp-grp-smoothvec}
  Let \((\pi,\cH)\) be a strongly continuous non-degenerate \(*\)-representation of \(S\) and let \((\pi_1^c,\cH)\) be the analytic continuation of \(\pi\) to \(G_1^c\) {\rm(Theorem \ref{thm:si-integration})}. Then, for every smooth vector \(v \in \cH\) of \((\pi,\cH)\), the set \(\pi(S)v\) consists of smooth vectors of \((\pi_1^c,\cH)\).
\end{prop}
\begin{proof}
  Recall that, for \(x \in \fq\), we have \(\p\pi_1^c(ix) = i\cL_x^\pi\) (cf.\ Remark \ref{rem:rkhs-ident-noncyclic}).

  Let \(v \in \cH\) such that the orbit map \(\pi^v : S \rightarrow \cH\) is smooth. Let \(B = \{x_1,\ldots,x_\ell\} \subset \fq\) be a basis of \(\fq\). Using coordinates of the second kind, we can find, for all \(s \in S\) and all \(x_{j_1},\ldots,{x_{j_n}} \in B\), an \(\varepsilon > 0\) such that
  \[\exp(t_1 x_{j_1})\ldots\exp(t_n x_{j_n})s \in S, \quad \text{for } |t_1| < \varepsilon,\ldots,\, |t_n| < \varepsilon.\]
  For \(n \in \N\), we prove by induction over \(n\) that \(\pi(s)v \in \cD(\cL^\pi_{x_{j_1}},\ldots,\cL^\pi_{x_{j_n}})\) and
  \begin{equation}
    \label{eq:semigrp-grp-smoothvec}
    \frac{\partial^n}{\partial t_n \ldots \partial t_1}\Big\lvert_{t_1 = \ldots = t_n = 0} \pi(\exp(t_1 x_{j_1})\ldots \exp(t_n x_{j_n})s)v = \cL^\pi_{x_{j_1}}\ldots\cL^\pi_{x_{j_n}}\pi(s)v 
  \end{equation}
  for all \(x_{j_1},\ldots,x_{j_n} \in B\) and all \(s \in S\).

  For \(n = 1\), this follows from Proposition \ref{prop:semgrp-rep-curve}(a) and \((\cL^\pi_x\lvert_{\cH^0})^* = \cL^\pi_x\) (cf.\ Remark \ref{rem:rkhs-ident-noncyclic}). For \(n > 1\), consider as above the map
  \[\eta : (-\varepsilon,\varepsilon)^n \rightarrow \cH, \quad \eta(t_1,\ldots,t_n) := \pi(\exp(t_1 x_{j_1})\ldots\exp(t_n x_{j_n})s)v,\]
  for some \(\varepsilon > 0\). Since \(v\) is a smooth vector for \(\pi\), the map \(\eta\) is smooth. In particular, the map
  \[\tilde \eta : (-\varepsilon,\varepsilon) \rightarrow \cH, \quad \tilde \eta (t) := \frac{\partial^{n-1}}{\partial t_{n-1}\ldots\partial t_1}\Big\lvert_{t_{n-1} = \ldots = t_1 = 0} \eta(t_1,\ldots,t_{n-1},t)\]
    is differentiable with
    \[\bigderat{0}\tilde\eta(t) \overset{\eqref{eq:semigrp-grp-smoothvec}}{=} \bigderat{0} \cL^\pi_{x_{j_1}}\ldots\cL^\pi_{x_{j_{n-1}}}\pi(\exp(tx_{j_n})s)v\]
    by induction. Since this limit and the limit \(\derat{0} \cL^\pi_{x_{j_2}}\ldots\cL^\pi_{x_{j_{n-1}}}\pi(\exp(tx_{j_n})s)v\) both exist, the closedness of \(\cL^\pi_{x_{j_1}}\) and the induction hypothesis imply that
    \begin{align*}
      \bigderat{0}\tilde\eta(t) &= \cL^\pi_{x_{j_1}}\bigderat{0}\cL^\pi_{x_{j_2}}\ldots\cL^\pi_{x_{j_{n-1}}}\pi(\exp(tx_{j_n})s)v \\
                             &= \cL^\pi_{x_{j_1}}\frac{\partial^{n-1}}{\partial t\,\partial t_{n-1}\ldots\partial t_2}\Big\lvert_{t=t_{n-1}=\ldots=t_2=0} \pi(\exp(t_2 x_{j_2})\ldots\exp(t_{n-1} x_{j_{n-1}})\exp(t x_{j_n})s)v \\
                             &= \cL^\pi_{x_{j_1}}\ldots\cL^\pi_{x_{j_n}}\pi(s)v.
  \end{align*}
  This proves \eqref{eq:semigrp-grp-smoothvec}. Hence, we have
  \[\pi(s)v \in \bigcap_{n \in \N,\,x_{j_k} \in B} \cD(\partial \pi_1^c(ix_{j_1})\ldots \partial \pi_1^c(ix_{j_n})) \quad \text{for all } s \in S.\]
  Since \(iB\) generates \(\g_1^c\) in the sense of Lie algebras, the claim now follows from Proposition \ref{prop:smoothvec-generators}.
\end{proof}

\begin{thm}
  \label{thm:semigrp-rep-int}
  Let \((\pi,\cH)\) be a strongly continuous non-degenerate \(*\)-representation of \(S\) and let \((\pi_1^c,\cH)\) be the analytic continuation of \(\pi\) to \(G_1^c\) {\rm(Theorem \ref{thm:si-integration})}. Then the following holds:
  \begin{enumerate}
    \item For \(y \in \fq\) and \(s \in S\) with \(\exp(ty)s \in S\) for \(|t| < \varepsilon\), we have
      \begin{equation}
        \label{eq:thm-int-analytic-curve}
        \pi(\exp(ty)s) = e^{-it \p\pi_1^c(iy)}\pi(s) \quad \text{for } |t| < \varepsilon.
      \end{equation}
      The curve \eqref{eq:thm-int-analytic-curve} is analytic with respect to \(t\) as a \(B(\cH)\)-valued curve.
    \item For \(x \in [\fq,\fq] \subset \fh\) and \(s \in S\) with \(\exp(tx)s\) for \(|t| < \varepsilon\), we have
      \begin{equation}
        \label{eq:thm-int-commut-curve}
        \pi(\exp(tx)s) = \pi_1^c(\exp(tx))\pi(s) \quad \text{for } |t| < \varepsilon.
      \end{equation}
  \end{enumerate}
\end{thm}
\begin{proof}
  (a) Let \(y \in \fq\), \(s \in S\) and \(\varepsilon > 0\) such that \(\exp(ty)s \in S\) for \(|t| < \varepsilon\). Consider the curve
  \[F : (-\varepsilon,\varepsilon) \rightarrow B(\cH), \quad F(t) := \pi(\exp(ty)s).\]
  Recall from Remark \ref{rem:rkhs-ident-noncyclic} that \(\partial\pi_1^c(iy) = i\cL_y^\pi\). Then Proposition \ref{prop:semgrp-rep-curve} shows that
  \[\pi(s)\cH \subset \cD(e^{t\cL_y^\pi}) \quad \text{and} \quad F(t) = e^{t\cL_y^\pi}\pi(s) = e^{-it\partial\pi_1^c(iy)}\pi(s), \quad |t| < \varepsilon.\]
  This proves \eqref{eq:thm-int-analytic-curve}. The analyticity of \(F\) follows from Lemma \ref{lem:holom-exp-prod}.

  (b) Let \(x \in [\fq,\fq],\, s \in S,\) and \(\varepsilon > 0\) such that \(\exp(tx)s \in S\) for \(|t| < \varepsilon\). We show that \eqref{eq:thm-int-commut-curve} holds on the dense subspace \(\cH^0 = \spann \{\pi(f)\cH : f \in C_c^\infty(S)\}\) (cf.\ Remark \ref{rem:rkhs-ident-noncyclic}). Therefore, let \(v \in \cH^0\). Consider the curve
  \[\gamma: (-\varepsilon, \varepsilon) \rightarrow \cH, \quad \gamma(t) := \pi(\exp(tx)s)v.\]
  By Lemma \ref{lem:semigrp-smooth-vec}, \(v\) is a smooth vector of \((\pi,\cH)\). Thus, by Proposition \ref{prop:semgrp-rep-curve}, \(\gamma'(t) = (\cL^\pi_{-x}\lvert_{\cH^0})^*\gamma(t)\). Proposition \ref{prop:semigrp-grp-smoothvec} implies that \(\gamma((-\varepsilon,\varepsilon))\) consists of smooth vectors of \((\pi_1^c,\cH)\). In particular, we have \(\gamma(-\varepsilon,\varepsilon) \subset \cD(\partial \pi_1^c(x))\). By Remark \ref{rem:si-integration}, \(-(\cL^\pi_x\lvert_{\cH^0})^*\) is an extension of \(\partial \pi_1^c(x)\). Hence, we have \(\gamma'(t) = \partial \pi_1^c(x)\gamma(t)\).

  On the other hand, Proposition \ref{prop:semigrp-grp-smoothvec} implies that the curve
  \[\eta : \R \rightarrow \cH, \quad t \mapsto \pi_1^c(\exp(tx))\pi(s)v,\]
  is differentiable with \(\eta(0) = \pi(s)v\) and \(\eta'(t) = \partial \pi_1^c(x)\eta(t)\). Since \(\gamma\) and \(\eta\) are both solutions to the same initial value problem on \((-\varepsilon,\varepsilon)\), we have \(\gamma(t) = \eta(t)\) for \(|t| < \varepsilon\) by Stone's Theorem. This implies \eqref{eq:thm-int-commut-curve}.
\end{proof}

\subsection{Local representations}

Our goal in this section is to extend the representation obtained in Theorem \ref{thm:si-integration} to a unitary representation of the 1-connected Lie group \(G^c\) with Lie algebra \(\g^c = \fh \oplus \fq\).
As the Lie subalgebra \(\g_1^c = [\fq,\fq] \oplus i\fq\) is an ideal in \(\g\), the integral subgroup \(\la \exp_{G^c}(\g_1^c)\ra\) of \(G^c\) with Lie algebra \(\g_1^c\) is normal, which implies that it is 1-connected (cf.\ \cite[Ch.\ XII, Thm.\ 1.2]{Ho65}). Therefore, we can identify \(G_1^c\) with a closed subgroup of \(G^c\).

Let \(H := G_0^\tau\). For \(h \in H\) and \(B \subset H\), we define
\[S_h = \{s \in S : hs \in S\} \quad \text{and} \quad S_B = \bigcap_{h \in B} S_h.\]
Note that \(S_h\) is open in \(G\) and that \(S_BS \subset S_B\).

Throughout this section, let \((\pi,\cH)\) be a strongly continuous non-degenerate \(*\)-representation of \(S\). For a subset \(N \subset \cH\), we say that \(N\) is \emph{total in \(\cH\)} if \(\spann \,N\) is dense in \(\cH\).

\begin{lemma}
  \label{lem:localrep-H-op}
  Suppose that for \(h \in H\) the subsets \(\pi(S_h)\cH\) and \(\pi(hS_h)\cH\) are total in \(\cH\).
  Then there exists a unique unitary operator \(\pi^H(h) \in \U(\cH)\) such that
  \begin{equation}
    \label{eq:localrep-H-op}
    \pi(hs) = \pi^H(h)\pi(s) \quad \text{for all } s \in S_h.
  \end{equation}
\end{lemma}
\begin{proof}
  For any \(s,t \in S_h\) and \(v,w \in \cH\), we have
  \[\la \pi(hs)v, \pi(ht) w \ra = \la \pi((ht)^*hs)v, w \ra = \la \pi(t)^*\pi(s)v, w \ra = \la \pi(s)v, \pi(t)w \ra.\]
  Hence, we obtain a linear isometry
  \[\spann(\pi(S_h)\cH) \rightarrow \spann(\pi(hS_h)\cH), \quad \sum_{i=1}^n \pi(s_i)v_i \mapsto \sum_{i=1}^n \pi(hs_i)v_i,\]
  which extends to a unitary operator \(\pi^H(h)\). Since \(\pi(S_h)\cH\) is total in \(\cH\), the operator \(\pi^H(h)\) is uniquely determined by \eqref{eq:localrep-H-op}.
\end{proof}

\begin{definition}
  \label{def:local-h-comp}
  Let \((G,\tau)\) be a symmetric Lie group and \(S \subset G\) be an open \(*\)-subsemigroup. Let \(H\) be the integral subgroup of \(G\) with Lie algebra \(\fh\). A strongly continuous non-degenerate \(*\)-representation \((\pi,\cH)\) of \(S\) is called \emph{locally \(H\)-compatible} if there exists a symmetric open \(\1\)-neighborhood \(V \subset H\) such that \(\pi(S_V)\cH\) is total in \(\cH\).
\end{definition}

The uniqueness of the unitary operators which were constructed in Lemma \ref{lem:localrep-H-op} implies that locally \(H\)-compatible representations yield ``local'' representations of \(H\) on \(\cH\):

\begin{prop}
  \label{prop:localrep-H-rep}
  Suppose that \((\pi,\cH)\) is locally \(H\)-compatible and let \(V \subset H\) be a symmetric open \(\1\)-neighborhood such that \(\pi(S_V)\cH\) is total in \(\cH\). Then the map
  \[\pi^H: V \rightarrow \U(\cH), \quad h \mapsto \pi^H(h),\]
  is strongly continuous and satisfies
  \[\pi^H(g)\pi^H(h) = \pi^H(gh), \quad \pi^H(g)^* = \pi^H(g^{-1}), \quad \text{for all } g,h \in V \text{ with } gh \in V\]
  and
  \[\pi(hs) = \pi^H(h)\pi(s) \quad \text{for all } s \in S_h,\,h \in V.\]
\end{prop}
\begin{proof}
  Since \(V\) is symmetric, we have for each \(h \in V\) that \(S_V \subset S_h\) and \(S_V \subset S_{h^{-1}} = hS_h\). Hence, the premises of Lemma \ref{lem:localrep-H-op} are satisfied for each \(h \in V\), so that we obtain a map \(\pi^{H} : V \rightarrow \U(\cH)\). The map \(\pi^H\) is strongly continuous on \(\pi(S_V)\cH\) because of \eqref{eq:localrep-H-op} and the strong continuity of \((\pi,\cH)\). The other properties follow from the uniqueness of \(\pi^H\).
\end{proof}

\begin{examples}(Sufficient conditions for local \(H\)-compatibility)
  \label{ex:local-h-comp}
  \begin{inlinelist}
  \item If the semigroup \(S\) satisfies \(HS = S\), then every strongly continuous non-degenerate \(*\)-representation of \(S\) is locally \(H\)-compatible because \(\pi(S)\cH\) is total in \(\cH\).
    In \cite{MNO15}, it is shown that non-degenerate strongly continuous representations of such semigroups always lead to analytic continuations to \(G^c\).\\

  \item Suppose that there exists a compact set \(C \subset S\) such that \(\pi(C)\cH\) is total in \(\cH\). Then there exists a symmetric open \(\1\)-neighborhood \(B \subset H\) such that \(BC \subset S\), i.e.\ \(C \subset S_B\). Thus, \(\pi\) is locally \(H\)-compatible.\\

  \item Let \(x \in \fq\) be such that \(\exp(tx) \in S\) for all \(t > 0\). Then \(C := \{\exp(x)\}\) satisfies the conditions of (2), i.e.\ \(\pi(\exp(x))\cH\) is dense in \(\cH\) (cf.\ Corollary \ref{cor:semigrp-rep-liecone}).
  \end{inlinelist}
\end{examples}

Example \ref{ex:local-h-comp}(2) suggests that one way to show that a representation \((\pi,\cH)\) of \(S\) is locally \(H\)-compatible is to prove that the subset
\[S_\reg^\pi := \{s \in S : \pi(s)\cH \text{ is dense in \(\cH\) and } \pi(s) \text{ is injective}\}\]
is non-empty. The following Lemma suggests that this property is natural:

\begin{lem}
  The subset \(S_\reg^\pi \subset S\) is an open \(*\)-subsemigroup of \(S\).
\end{lem}
\begin{proof}
  That \(S_\reg^\pi\) is \(*\)-invariant follows from the fact that, for every \(s \in S\), the operator \(\pi(s)\) has dense range if and only if \(\pi(s)^*\) is injective. Since the product of two bounded injective operators with dense range is again an injective operator with dense range, the subset \(S_\reg^\pi\) is a \(*\)-subsemigroup of \(S\).

  In order to show that \(S_\reg^\pi\) is open in \(S\), we define for every \(s \in S\) the open subsets
  \[U_L(s) := S \cap sS^{-1}, \quad U_R(s) := S \cap S^{-1}s, \quad \text{and}  \quad U(s) := U_L(s) \cap U_R(s).\]
  We claim that, for \(s \in S_\reg^\pi\), we have \(U(s) \subset S_\reg^\pi\): If \(s' \in U_L(s)\), then there exists a factorization \(s's'' = s\), where \(s'' \in S\). Hence, \(\pi(s) = \pi(s')\pi(s'')\) implies that \(\pi(s)\cH = \pi(s')\pi(s'')\cH \subset \pi(s')\cH\), i.e.\ the range of \(\pi(s')\) is dense in \(\cH\). With a similar argument, we see that \(\pi(s')\) is injective for every \(s' \in U_R(s)\). As a result, \(U(s)\) is contained in \(S_\reg^\pi\) if \(s \in S_\reg^\pi\).

  Now \(s \in U(s^2) \subset S_\reg^\pi\) for every \(s \in S_\reg^\pi\) shows that \(S_\reg^\pi\) is open in \(S\).
\end{proof}

\begin{rem}
  \label{rem:local-h-comp-unique}
  Let \((\pi,\cH)\) be a locally \(H\)-compatible representation of \(S\) and, for \(k=1,2\), let \(B_k \subset H\) be a symmetric open \(\1\)-neighborhood such that \(\pi(S_{B_k})\cH\) is total in \(\cH\). Let \(\pi_k^H\) be the corresponding maps we obtain from Proposition \ref{prop:localrep-H-rep} when applied to \(V = B_k\). Then, for \(s \in S_{B_1 \cap B_2}\) and \(h \in B_1 \cap B_2\), we have
  \[\pi_1^H(h)\pi(s) \overset{\eqref{eq:localrep-H-op}}{=} \pi(hs) \overset{\eqref{eq:localrep-H-op}}{=} \pi_2^H(h)\pi(s).\]
  Since \(S_{B_k} \subset S_{B_1 \cap B_2}\) for \(k=1,2\), the subset \(\pi(S_{B_1 \cap B_2})\cH\) is total in \(\cH\), so that \(\pi_1^H\) and \(\pi_2^H\) coincide on \(B_1 \cap B_2\). 
\end{rem}

\begin{prop}
  \label{prop:localrep-comm-rel}
  Suppose that \((\pi,\cH)\) is locally \(H\)-compatible {\rm(Definition \ref{def:local-h-comp})} and let \((\pi_1^c,\cH)\) be the analytic continuation of \(\pi\) to \(G_1^c\) {\rm(Theorem \ref{thm:si-integration})}. Let \(q_H: \tilde H \rightarrow H\) be a universal covering of \(H\). Then there exists a unique strongly continuous unitary representation \(\pi^{\tilde H} : \tilde H \rightarrow \U(\cH)\) with the following properties:
  \begin{enumerate}
    \item Let \(B \subset \tilde H\) be a symmetric open \(\1\)-neighborhood such that \(\pi(S_{q_H(B)})\cH\) is total in \(\cH\). Then
      \[\pi^{\tilde H}(h)\pi(s) = \pi(q_H(h)s)\]
      for all \(h \in B, s \in S_{q_H(B)} = \{s \in S : q_H(B)s \subset S\}\).
    \item \(\pi^{\tilde H}(h)\pi_1^c(g)\pi^{\tilde H}(h)^{-1} = \pi_1^c(\alpha_h(g))\) for all \(h \in \tilde H,\, g \in G_1^c\), where \(\alpha_h \in \Aut(G_1^c)\) with \(\L(\alpha_h) = \Ad(h)\). In particular, the closed convex cone \[W := \oline{\{x \in \g_1^c : \sup\spec(i\p\pi_1^c(x)) < \infty\}}\] is \(\Ad(G^c)\)-invariant.
    \item \(\pi^{\tilde H}(h)e^{i \p\pi_1^c(x)}\pi^{\tilde H}(h)^{-1} = e^{i \p\pi_1^c(\Ad(h)x)}\) for \(h \in \tilde H\) and \(x \in \g_1^c\).
  \end{enumerate}
\end{prop}

\begin{proof}
  (a) The local \(H\)-compatibility of \(\pi\) implies that there exists a symmetric open \(\1\)-neighborhood \(B \subset H\) such that \(\pi(S_B)\cH\) is total in \(\cH\). Let \(\pi^H : B \rightarrow \U(\cH)\) be the corresponding local representation we obtain from Proposition \ref{prop:localrep-H-rep}.
  By the Monodromy Principle for Lie groups \cite[Prop.\ 9.5.8]{HN12}, we can lift and extend \(\pi^H\) to a continuous unitary representation \((\pi^{\tilde H},\cH)\) of \(\tilde H\) which satisfies (a).
  The representation is uniquely determined by (a) and, by Remark \ref{rem:local-h-comp-unique}, it does not depend on the choice of \(B\).

  (b) Let \(y \in \fq\), \(h \in B\), \(s \in S_h\), and \(v \in \cH\).
  Choose \(\varepsilon > 0\) such that \(\exp(ty)s \in S_h\) for all \(t \in (-\varepsilon,\varepsilon)\).
  By Theorem \ref{thm:semigrp-rep-int}(a), the continuous curves
  \[\alpha(t) := \pi^H(h)\pi_1^c(\exp(tiy))\pi(s)v \quad \text{and}\]
  \[\beta(t) := \pi_1^c(\exp(it\Ad(h)y))\pi^H(h)\pi(s)v = \pi_1^c(\exp(it\Ad(h)y))\pi(hs)v\]
  have analytic continuations to the strip \(\{z \in \C : -\varepsilon < \Im(z) < \varepsilon\}\). By evaluating at \(it\), for \(|t| < \varepsilon\), and applying \eqref{eq:thm-int-analytic-curve} and \eqref{eq:localrep-H-op}, we obtain
  \[\alpha(it) = \pi^H(h)e^{-it \p\pi_1^c(iy)}\pi(s)v = \pi^H(h)\pi(\exp(ty)s)v = \pi(h\exp(ty)s)v\]
  and
  \[\beta(it) = e^{-it \p\pi_1^c(i\Ad(h)y)}\pi(h\exp(y))v = \pi(\exp(t\Ad(h)y)h\exp(y))v = \alpha(it).\]
  Thus we also have \(\alpha(t) = \beta(t)\) for all \(t \in \R\). By applying the same argument to all \(s \in S_h\) and \(v \in \cH\), we obtain by the totality of \(\pi(S_h)\cH\) that
  \[\pi^{\tilde H}(h)\pi_1^c(\exp(ity))\pi^{\tilde H}(h)^{-1} = \pi_1^c(\exp(it\Ad(h)y)), \quad \text{for } t \in \R, y \in \fq, h \in \tilde H.\]
  Since \(i\fq\) generates \(\g_1^c\), this proves (b).

  (c) Let \(x \in \g_1^c\) and \(h \in \tilde H\). Then (b) implies \(\pi^{\tilde H}(h)\p\pi_1^c(x)\pi^{\tilde H}(h)^{-1} = \p\pi_1^c(\Ad(h)x)\). Thus, (c) follows from spectral calculus.
\end{proof}

\begin{lem}
  \label{lem:semidir-prod-surj}
  Let \((G,\tau)\) be a 1-connected symmetric Lie group with Lie algebra \(\g = \h \oplus \fq\) and let \(G_1\) be the integral subgroup of \(G\) with Lie algebra \(\g_1 = [\fq,\fq] \oplus \fq\). Let \(H\) be the integral subgroup of \(G\) with Lie algebra \(\fh\) and let \(q_H: \tilde H \rightarrow H\) be the universal covering group of \(H\). Consider the semidirect product \(G_1 \rtimes \tilde H\), where \(\tilde H\) acts on \(G_1\) by the integrated adjoint representation, and the map
  \begin{equation}
    \label{eq:semidir-prod-surj}
    \varphi : G_1 \rtimes \tilde H \rightarrow G, \quad (g,h) \mapsto gq_H(h).
  \end{equation}
  Then the following holds:
  \begin{enumerate}
    \item \(\varphi\) is a surjective homomorphism of Lie groups whose kernel is given by the integral subgroup \(\Delta(H_\fq)\) with Lie algebra \(\{(x,-x) : x \in [\fq,\fq]\}\).
    \item Let \(H_1\) be the integral subgroup of \(G\) with Lie algebra \([\fq,\fq]\). Then \(\varphi\) restricts to a surjective homomorphism \(\varphi\lvert_{H_1 \rtimes \tilde H}: H_1 \rtimes \tilde H \rightarrow H\) whose kernel is given by \(\Delta(H_\fq)\).
  \end{enumerate}
\end{lem}
\begin{proof}
  (a) The derivative of \(\varphi\) is given by \(\L(\varphi)(g,h) = g + h\) for \(g \in \g_1,h \in \fh\).
  Thus, \(\varphi\) is surjective because \(G\) is connected and \(\Delta(H_\fq) = (\ker \varphi)_0\).
  
  It remains to show that the kernel of \(\varphi\) is connected. We first note that, since \(\g_1\) is an ideal in \(\g\), the subgroup \(G_1\) is 1-connected (cf.\ \cite[Ch.\ XII, Thm.\ 1.2]{Ho65}). Hence, the semidirect product \(G_1 \rtimes \tilde H\) is also 1-connected.
  If \(\ker \varphi\) was not connected, then the map
  \[(G_1 \rtimes \tilde H)/(\ker \varphi)_0 \rightarrow G, \quad g(\ker \varphi)_0 \mapsto \varphi(g),\]
  would be a non-trivial covering of \(G\). Hence, we have \((\ker \varphi)_0 = \ker \varphi\).
  By restricting \(\varphi\) to the subgroup \(H_1 \rtimes \tilde H\), we obtain (b).
\end{proof}

\begin{thm}
  \label{thm:localrep}
  Suppose that \((\pi,\cH)\) is locally \(H\)-compatible {\rm(Definition \ref{def:local-h-comp})} and let \((\pi_1^c,\cH)\) be the analytic continuation of \(\pi\) to \(G_1^c\) {\rm(Theorem \ref{thm:si-integration})}. Let \(H^c\) be the integral subgroup of \(G^c\) with Lie algebra \(\fh\), let \(q_{H^c}: \tilde H \rightarrow H^c\) be its universal covering, and let \((\pi^{\tilde H}, \cH)\) be the unitary representation of \(\tilde H\) constructed in {\rm Proposition \ref{prop:localrep-comm-rel}}.
  Then there exists a unique extension of \((\pi^c_1,\cH)\) to a strongly continuous unitary representation \((\pi^c,\cH)\) of \(G^c\) such that \(\pi^c(q_{H^c}(h)) = \pi^{\tilde H}(h)\) for all \(h \in \tilde H\).
\end{thm}
\begin{proof}
  Our assumptions already determine \(\pi^c\) on \(H^c\) and \(G_1^c\), which proves the uniqueness of \(\pi^c\). It remains to show its existence.
  By Proposition \ref{prop:localrep-comm-rel}b), we can extend \((\pi_1^c,\cH)\) to a strongly continuous unitary representation
  \[\tilde \pi : G_1^c \rtimes \tilde H \rightarrow \U(\cH), \quad (g,h) \mapsto \pi_1^c(g)\pi^{\tilde H}(h),\]
  where \(\tilde H\) acts on \(G_1^c\) by the integrated adjoint representation.
  The map
  \[G_1^c \rtimes \tilde H \rightarrow G^c, \quad (g,h) \mapsto gq_{H^c}(h),\]
  is a surjective homomorphism of Lie groups whose kernel is the integral subgroup \(\Delta(H_\fq)\) of \(G_1^c \rtimes \tilde H\) with Lie algebra \(\{(x,-x) : x \in [\fq,\fq]\} \subset \g_1^c \rtimes \fh\) (cf.\ Lemma \ref{lem:semidir-prod-surj}(a)).
  Thus it remains to show that \((\tilde \pi, \cH)\) factors through a continuous unitary representation of the Lie group \((G_1^c \rtimes \tilde H)/\Delta(H_\fq) \cong G^c\). 
  Choose a symmetric open \(\1\)-neighborhood \(B \subset H\) such that \(\pi(S_B)\cH\) is dense in \(\cH\). Let \(x \in [\fq,\fq]\) and let \(\varepsilon > 0\) such that \(\exp(tx) \in B\) for \(|t| < \varepsilon\). 
  Then, by Theorem \ref{thm:semigrp-rep-int}, we have for every \(s \in S_B\)
  \[\pi^c_1(\exp_{G_1^c}(tx))\pi(s) = \pi(\exp_G(tx)s) \quad \text{for } |t| < \varepsilon.\]
  We thus have \(\pi^c_1(\exp_{G_1^c}(tx)) = \pi^{\tilde H}(\exp_{\tilde H}(tx))\) for \(t \in (-\varepsilon,\varepsilon)\) because \(\pi(S_B)\cH\) is total in \(\cH\), which implies that equality holds for all \(t \in \R\).
  As a result, we have \(\Delta(H_\fq) \subset \ker \tilde\pi\), which proves the claim.
\end{proof}

We call \((\pi^c,\cH)\) the \emph{analytic continuation of \((\pi,\cH)\) to \(G^c\)}. The following theorem explains the relation between the semigroup representation and its analytic continuation.

\begin{thm}{\rm (Analytic Continuation Theorem)}
  \label{thm:localrep-analytic-cont}
  Let \(S \subset G\) be an open \(*\)-subsemigroup of \(G\) and let \((\pi,\cH)\) be a continuous non-degenerate \(*\)-representation of \(S\) which is locally \(H\)-compatible {\rm(Definition \ref{def:local-h-comp})}. Then the analytic continuation \((\pi^c,\cH)\) of \((\pi,\cH)\) to \(G^c\) {\rm(Theorem \ref{thm:localrep})} has the following properties:
  \begin{enumerate}
    \item For \(y \in \fq\) and \(s \in S\) with \(\exp(ty)s \in S\) for \(|t| < \varepsilon\), we have
      \[\pi(\exp(ty)s) = e^{-it \p\pi^c(iy)}\pi(s) \quad \text{for }\, |t| < \varepsilon.\]
      The curve \(t \mapsto \pi(\exp(ty)s)\) is analytic as a \(B(\cH)\)-valued curve.
    \item For \(x \in [\fq,\fq]\) and \(s \in S\) with \(\exp(tx)s \in S\) for \(|t| < \varepsilon\), we have
      \[\pi(\exp(tx)s) = \pi^c(\exp(tx))\pi(s) \quad \text{for }\, |t| < \varepsilon.\]
    \item If \(B \subset H\) is a symmetric open \(\1\)-neighborhood such that \(\pi(S_B)\cH\) is total in \(\cH\), then 
      \[\pi(\exp_G(x)s) = \pi^c(\exp_{G^c}(x))\pi(s) \quad \text{for all } s \in S_B, x \in \exp_{G}^{-1}(B).\]
  \end{enumerate}
\end{thm}
\begin{proof}
  Properties (a) and (b) follow from Theorem \ref{thm:semigrp-rep-int} and property (c) is a consequence of Proposition \ref{prop:localrep-comm-rel}(a).
\end{proof}

For the remainder of this section, we will consider a certain class of semigroups for which the local \(H\)-compatibility is always satisfied.
For the open semigroup \(S\), we define
\[\L^o(S) := \{x \in \g : (\forall t > 0) \, \exp(tx) \in S\}\]
and suppose that \(\L^o(S) \cap \fq\) is non-empty.

\begin{rem}
  \label{rem:semgrp-rep-locbnd}
  Every strongly continuous representation \((\pi, \cH)\) of \(S\) is locally bounded: For every \(s \in S\), there exists a compact neighborhood \(V \subset S\) of \(s\), so that the map \(s \mapsto \|\pi(s)v\|\) is bounded on \(V\) for every \(v \in \cH\).
    By the Principle of Uniform Boundedness, this implies that \(\sup_{t \in V}\|\pi(t)\| < \infty\).
\end{rem}

\begin{prop}
  \label{prop:cont-op-semgrps}
  Let \(c > 0\) and \(\pi: (c,\infty) \rightarrow B(\cH)\) be a locally bounded non-degenerate representation by selfadjoint operators. Then there there exists a unique extension to a representation \(\pi: [0,\infty) \rightarrow B(\cH)\) by selfadjoint operators with \(\pi(0) = \id_\cH\). The representation \(\pi\) is analytic on \((0,\infty)\) and, for every \(t \in [0,\infty)\), the subspace \(\pi(t)\cH\) is dense in \(\cH\).
\end{prop}
\begin{proof}
  By \cite[Lemma VI.2.2]{Ne00}, \(\pi\) can be uniquely extended to a representation of \([0,\infty)\) which is strongly continuous on \((0,\infty)\) and satisfies \(\pi(0) = \id_\cH\). From the proof it also follows that the extended representation is selfadjoint. Hence, the generator \(A\) of \(\pi\) is selfadjoint and we have \(\pi(t) = e^{tA},\,t \geq 0,\) in the sense of spectral calculus.

  For \(v \in \cH\), consider the continuous function
  \[\varphi(t) := \la v, \pi(t)v \ra, \quad t > 0.\]
  The kernel \(K(t,t') := \varphi(\frac{t+t'}{2}) = \la \pi(t/2)v, \pi(t'/2)v\ra,\,t,t' > 0\) is positive definite, hence the function \(\varphi\) is analytic in \((0,\infty)\) by \cite{Wi34}. By \cite[Thm.\ 5.1]{Ne10b}, \(\pi\) is strongly analytic in \((0,\infty)\). The local boundedness of \(\pi\) (cf.\ Remark \ref{rem:semgrp-rep-locbnd}) implies that it is also analytic as a \(B(\cH)\) valued map (cf.\ \cite[Cor.\ A.III.5]{Ne00}).

  Let \(t \in [0,\infty)\). It remains to show that \(\pi(t)\cH\) is dense in \(\cH\). For \(m > 0\), we define \(E_m := \chi_{[-m,m]}(A)\) by spectral calculus and set
  \[\pi_m(t) := E_m\pi(t) = \pi(t)E_m.\]
  Since \(\pi_m(t)\) is invertible for each \(m\) and \(\lim_{m \rightarrow \infty}E_m = \id_\cH\) strongly, we have for all \(v \in \cH\):
  \[\lim_{m \rightarrow \infty}\pi(t)\pi_m(-t)v = \lim_{m \rightarrow \infty}\pi_m(t)\pi_m(-t)v = v.\]
  Hence, \(\pi(t)\cH\) is dense in \(\cH\).
\end{proof}

\begin{coro}
  \label{cor:semigrp-rep-liecone-curve}
  Let \(y \in \fq \cap \L^o(S)\). Then the curve
  \[\gamma : (0,\infty) \rightarrow B(\cH), \quad t \mapsto \pi(\exp(ty))\]
  is strongly continuous and analytic on \((0,\infty)\). For each \(t > 0\), the subspace \(\gamma(t)\cH\) is dense in \(\cH\).
\end{coro}
\begin{proof}
  For all \(s \in S\) and \(v \in \cH\), we have
  \[\lim_{t \rightarrow 0} \pi(\exp(ty))\pi(s)v = \lim_{t \rightarrow 0}\pi(\exp(ty)s)v = \pi(s)v,\]
  which shows that \(\gamma\) is a non-degenerate representation because \(\pi\) is non-degenerate. As \(\pi\) is locally bounded (cf.\ Remark \ref{rem:semgrp-rep-locbnd}), \(\gamma\) is locally bounded as well. Hence, the claim follows from Proposition \ref{prop:cont-op-semgrps}.
\end{proof}

\begin{coro}
  \label{cor:semigrp-rep-liecone}
  If \(\fq \cap \L^o(S) \neq \emptyset\), then every strongly continuous non-degenerate \(*\)-representation of \(S\) is locally \(H\)-compatible. 
\end{coro}
\begin{proof}
  This is a direct consequence of Corollary \ref{cor:semigrp-rep-liecone-curve} and Example \ref{ex:local-h-comp}(3).
\end{proof}

The analytic continuation of the semigroup representation is already determined by the values of the semigroup representation on open sets in \(S\) with a non-empty intersection with \(\L^o(S) \cap \fq\):

\begin{prop}
  \label{prop:analytic-cont-undet}
  For \(k \in \{a,b\}\), let \((S_k,*) \subset G\) be open \(*\)-subsemigroups of \(G\). Let \((\pi_k,\cH)\) be continuous non-degenerate representations of \((S_k,*)\) and suppose that there exists \(y_0 \in \L^o(S_a) \cap \L^o(S_b) \cap \fq\) and a neighborhood \(V \subset S_a \cap S_b\) of \(\exp(y_0)\) such that \(\pi_a\lvert_V = \pi_b\lvert_V\). Then the analytic continuations \((\pi_k^c,\cH)\) of \(\pi_k\) to \(G^c\) obtained from {\rm Theorem \ref{thm:localrep}} coincide.
\end{prop}
\begin{proof}
  Since \(G^c\) is connected, it suffices to show that the one-parameter groups of \(\pi_a^c\) and \(\pi_b^c\) coincide.
  Set \(s := \exp(y_0)\). Let \(y \in \fq\) and choose \(\varepsilon > 0\) such that \(\exp(ty)s \in V\) for \(|t| < \varepsilon\). Then we have by Theorem \ref{thm:localrep-analytic-cont}
  \[e^{-it \p\pi^c_a(iy)}\pi_a(s) = \pi_a(\exp(ty)s) = \pi_b(\exp(ty)s) =e^{-it \p\pi^c_b(iy)} \pi_b(s) = e^{-it \p\pi^c_b(iy)} \pi_a(s)\]
  for \(|t| < \varepsilon\). Since this curve is analytic with respect to \(t\) (cf.\ Theorem \ref{thm:localrep-analytic-cont}), we have
  \[\pi_a^c(\exp(ity))\pi_a(s) = \pi_b^c(\exp(ity))\pi_b(s) = \pi_b^c(\exp(ity))\pi_a(s), \quad t \in \R.\]
  Now the density of the subspace \(\pi_a(s)\cH\) (cf.\ Corollary \ref{cor:semigrp-rep-liecone-curve}) implies that \(\pi_a^c(\exp(ity)) = \pi_b^c(\exp(ity))\) for all \(t \in \R\).

  Let now \(x \in \fh\) and let \(B \subset H\) be a symmetric open \(\1\)-neighborhood such that \(Bs \subset V\). Choose \(\varepsilon > 0\) such that \(\exp(tx) \in B\) for \(|t| < \varepsilon\). Then we have
  \[\pi_a^c(\exp(tx))\pi_a(s) = \pi_a(\exp(tx)s) = \pi_b(\exp(tx)s) = \pi_b^c(\exp(tx))\pi_b(s) = \pi_b^c(\exp(tx))\pi_a(s), \,\, |t| < \varepsilon.\]
  The density of \(\pi_a(s)\cH\) implies that \(\pi_a^c(\exp(tx)) = \pi_b^c(\exp(tx))\) for \(|t| < \varepsilon\). Thus the same holds for all \(t \in \R\). This shows \(\pi_a^c = \pi_b^c\).
\end{proof}

\section{Extensions to semigroup representations}
\label{sec:ext-semgrp-rep}

In the previous section, we have shown that strongly continuous non-degenerate semigroup representations of \(S\) have an analytic extension to \(G^c\) if there exists \(y \in \fq\) such that \(\exp(ty) \in S\) for all \(t > 0\). In this section, we show that the semigroup representation further extends to a representation of a certain generalization of an Olshanski semigroup if \(\L^o(S) \cap \fq\) has inner points.

\subsection{Invariant cones in Lie algebras and Olshanski semigroups}

We fix the following notation: Let \(V\) be a finite dimensional real vector space. A closed convex cone \(W \subset V\) is called a \emph{wedge}. We define \(H(W) := W \cap (-W)\) as the \emph{edge of the wedge} \(W\). We say that \(W\) is \emph{pointed} if \(H(W) = \{0\}\) and that it is \emph{generating} if \(W - W = V\). For a subset \(E \subset V\), we define \(B(E) := \{\omega \in V^* : \inf \omega(E) > -\infty\}\). Furthermore, we denote by \(E^o\) the interior of \(E\).

A wedge \(W \subset \g\) in a Lie algebra \(\g\) is called \emph{invariant} if \(e^{\ad x}W = W\) for all \(x \in \g\). Note that in this case, the subspaces \(H(W)\) and \(W - W\) are ideals in \(\g\).

The following example is especially important in the context of unitary representation theory: Let \((\pi,\cH)\) be a strongly continuous unitary representation of a connected Lie group \(G\) and let \(\bP^\infty := \bP(\cH^\infty)\) be the projective space of the space of smooth vectors of \(\pi\) (cf.\ Appendix \ref{sec:diffvectors}).
The \emph{convex momentum set \(I_\pi\) of \(\pi\)} is defined as the closed convex hull of the image of the map
\[\Phi: \bP^\infty \rightarrow \g^*, \quad \Phi([v])(x) := \frac{\la \dd \pi(x)v,v\ra}{i\la v, v \ra}.\]
By \cite[Lem.\ X.1.6]{Ne00}, we have \(B(I_\pi) = \{x \in \g : \sup \spec (i \p\pi (x)) < \infty\}\).
Moreover, \(B(I_\pi)\) is a convex \(\Ad(G)\)-invariant cone.
We define \(W_\pi := B(I_\pi)^o\).

\begin{definition}
  Let \(\g\) be a Lie algebra and let \(x \in \g\). Then \(x\) is called \emph{weakly elliptic} if \(\spec(\ad x) \subset i\R\). A subset \(W \subset \g\) is called \emph{weakly elliptic} if it consists of weakly elliptic elements.
  We say that \(x\) is \emph{weakly hyperbolic} if \(\spec(\ad x) \subset \R\) and we call a subset \(W \subset \g\) \emph{weakly hyperbolic} if it consists of weakly hyperbolic elements.
\end{definition}

\begin{examples}\label{ex:hyp-ell-sets}
  \begin{inlinelist}
  \item Let \((\pi,\cH)\) be a strongly continuous unitary representation of \(G\) with a discrete kernel. Then the convex cone \(B(I_\pi) \subset \g\) is weakly elliptic (cf.\ \cite[Rem.\ XI.2.4]{Ne00}).\\

  \item Let \(W \subset \g\) be a pointed \(e^{\ad \g}\)-invariant wedge. Then \(W\) is weakly elliptic (cf. \cite[p.\ 196]{HN93}).\\
  
  \item Let \(iW \subset \g^c\) be a pointed \(e^{\ad \g^c}\)-invariant wedge. Then \(C := W \cap \fq\) is weakly hyperbolic.
  \end{inlinelist}
\end{examples}

We recall some basic facts about tangent wedges of subsemigroups of Lie groups: For a closed subsemigroup \(S\) of a connected Lie group \(G\), we define the \emph{tangent wedge} of \(S\) by
\[\L(S) := \{x \in \L(G) : \exp(tx) \in S \text{ for all } t \geq 0\}.\]
For an \(\Ad(G)\)-invariant wedge \(W \subset \g\), we define \(S_W := \oline{\la \exp(W)\ra} \subset G\) as the closed subsemigroup generated by \(W\). We say that \(W\) is \emph{global} if \(\L(S_W) = W\) which is equivalent to \(W = \L(S)\) for a closed subsemigroup \(S \subset G\).

\begin{rem}
  \label{rem:lietanwed-closure}
  Let \(S \subset G\) be an open \(*\)-subsemigroup of the symmetric Lie group \((G,\tau)\). Then the closure \(\oline S\) of \(S\) is a closed \(*\)-subsemigroup of \(G\). In particular, \(\L(\oline S)\) is a closed convex \((-\L(\tau))\)-invariant cone. Suppose now that \(\L(\oline S)\) has interior points in \(\g = \L(G)\), i.e. \(\L(\oline S) - \L(\oline S) = \L(G)\). Then \(p(x) := \frac{1}{2}(x - \L(\tau)(x)), x \in \g,\) is a projection onto \(\fq\) and we have 
  \[p(\L(\oline S)) = \L(\oline S) \cap \fq \quad \text{and} \quad (\L(\oline S) \cap \fq)^o = p(\L(\oline S)^o) = \L(\oline S)^o \cap \fq\]
  (cf.\ \cite[Prop.\ 1.6]{HN93}). Furthermore, we even have \(\L(\oline S)^o \subset \L^o(S)\), which can be seen as follows: Since \(\1 \in \oline S\), we have \((\oline S)^o = S\) (cf.\ \cite[Lem.\ 3.7(ii)]{HN93}). Moreover, since the exponential function \(\exp: \g \rightarrow G\) of \(G\) is regular on a 0-neighborhood, we have \(\exp(\L(\oline S)^o) \subset S\) because \(\L(\oline S)^o\) is a cone and \(S = (\oline S)^o\) is a semigroup ideal of \(\oline S\) (cf.\ \cite[Lem.\ 3.7(i)]{HN93}). In particular, \(\L(\oline S)^o \cap \fq \subset \L^o(S) \cap \fq\), so that all strongly continuous non-degenerate \(*\)-representations of \(S\) satisfy the local \(H\)-compatibility condition which is needed for the Analytic Continuation Theorem \ref{thm:localrep-analytic-cont} (cf.\ Corollary \ref{cor:semigrp-rep-liecone}).
\end{rem}

\begin{thm}{\rm (Lawson's Theorem on Olshanski Semigroups, \cite[Thm.\ XI.1.10]{Ne00})}
  \label{thm:Lawson}
  Let \((G,\tau)\) be a 1-connected symmetric Lie group, \(\g = \fh \oplus \fq\) be the associated symmetric Lie algebra, and \(C \subset \fq\) be an \(e^{\ad \fh}\)-invariant weakly hyperbolic closed convex cone. Then the set \(\Gamma(C) := G^\tau \exp(C)\) is a connected closed subsemigroup of \(G\) for which the polar map
  \begin{equation}
    \label{eq:olsh-polar}
  G^\tau \times C \rightarrow \Gamma(C), \quad (h,x) \mapsto h\exp(x),
\end{equation}
  is a homeomorphism.
\end{thm}

The subsemigroup \(\Gamma(C)\) is called an \emph{Olshanski semigroup}.
If \(\g\) is a complex Lie algebra and \(\fq = i\fh\), we call \(\Gamma(C)\) a \emph{complex Olshanski semigroup}.
The semigroup \(\Gamma(C)\) is a Lie subsemigroup of \(G\) with \(\L(\Gamma(C)) = \fh + C\) (cf.\ \cite[Cor.\ 7.35]{HN93}).
Hence, there exists a universal covering \(q: \tilde \Gamma(C) \rightarrow \Gamma(C)\) such that \(q\) is a homomorphism of topological monoids.
By lifting the polar decomposition \eqref{eq:olsh-polar}, we see that \(\tilde \Gamma(C)\) is homeomorphic to \(\tilde G^\tau \times C\), where \(\tilde G^\tau\) is the universal covering group of \(G^\tau\).
If \(H\) is a connected Lie group with Lie algebra \(\fh\), then there exists a discrete central subgroup \(D \subset \tilde G^\tau\) such that \(H \cong \tilde G^\tau/D\), and \(D\) is a discrete central subgroup of \(\tilde \Gamma(C)\).
We define \(\Gamma_H(C) := \tilde \Gamma(C)/D\).

For our purposes, we need a more general version of an Olshanski semigroup.
Thus, we drop the condition that \(C\) is weakly hyperbolic (respectively weakly elliptic in the complex case) and only assume that it is an \(e^{\ad \fh}\)-invariant wedge.

We look at the complex case first.

\begin{theorem}
  \label{thm:Ol-cplx}
  Let \(G\) be a 1-connected Lie group with Lie algebra \(\g\) and let \(\eta : G \rightarrow G_\C\) be its universal complexification. 
  Let \(W \subset \g\) be an \(e^{\ad \g}\)-invariant wedge and let \(\fn := H(W)\).
 Then the following holds:
 \begin{enumerate}
   \item The wedge \(\g + iW\) is global in \(G_\C\) and \(\Gamma_G(W) := \oline{\la \exp_{G_\C}(\g + iW)\ra} = N_\C\eta(G)\exp(iW) \subset G_\C\), where \(N_\C\) is the integral subgroup of \(G_\C\) with \(\L(N_\C) = \fn_{\C}\).
   \item The quotient map \(q: G_\C \rightarrow G_\C/N_\C\) maps \(\Gamma_G(W)\) onto the complex Olshanski semigroup \(\Gamma_Q(W')\), where \(Q\) is the integral subgroup of \(G_\C/N_\C\) with Lie algebra \(\g/\fn\) and \(W' = W/\fn\).
   \item The unit group \(\Gamma_G(W)^\times\) is given by \(N_\C\eta(G)\).
   \item For every closed convex cone \(W_1 \subset W\) with \(W = W_1 \oplus \fn\), the map
 \[p : W_1 \times \Gamma_G(W)^\times \rightarrow \Gamma_G(W), \quad (x,g) \mapsto \exp_{G_\C}(ix)g,\]
 is a homeomorphism.
 \end{enumerate}
\end{theorem}
\begin{proof}
The subspace \(\fn = W \cap -W\) is an ideal in \(\g\) and \(\fn_\C\) is an ideal in \(\g_\C\).
We also note that the universal complexification \(G_\C\) of \(G\) is 1-connected by \cite[Thm.\ 15.1.4]{HN12}.
Moreover, the integral subgroup \(N_\C\) of \(G_\C\) with Lie algebra \(\fn_\C\) is normal, hence 1-connected, and \(G_\C/N_\C\) is 1-connected as well by \cite[Ch.\ XII, Thm.\ 1.2]{Ho65}.

By \cite[Cor.\ 7.36]{HN93}, the wedge \(\g + iW\) is global in \(G_\C\) and \(\Gamma_G(W) = N_\C\eta(G)\exp(iW)\), which proves (a).
In particular, the quotient map \(q: G_\C \rightarrow G_\C/N_\C\) maps \(\Gamma_G(W)\) onto the Olshanski semigroup \(\Gamma_Q(W') \subset G_\C/N_\C\), where \(W' = W/\fn\) and \(Q\) is the integral subgroup of \(G_\C/N_\C\) with Lie algebra \(\g/\fn\).
Since the unit group of \(\Gamma_Q(W')\) is given by \(Q\exp(iW' \cap -iW') = Q\) (cf.\ \cite[Thm.\ XI.1.12]{Ne00}), we have \(\Gamma_G(W)^\times = q^{-1}(Q) = N_\C\eta(G)\).

It remains to show that \(p\) is a homeomorphism.
In view of \(q(\exp(iW)) \subset \exp(iW_1)N_\C\), we have
\[q^{-1}(\Gamma_Q(W')) = \Gamma_G(W) \subset N_\C\eta(G)\exp(iW_1) = \exp(iW_1)N_\C\eta(G),\]
hence \(p\) is surjective. 
Consider \(x,y \in W_1\) and \(g,h \in \eta(G)N_\C\) such that \(\exp(ix)g = \exp(iy)h\).
Then we have \(q(\exp(ix)g) = \exp(ix)gN_\C = \exp(iy)hN_\C\).
By using the polar decomposition of \(\Gamma_Q(W')\) (cf.\ Theorem \ref{thm:Lawson}), we see that this implies \(x + \fn = y + \fn\), i.e., \(x = y\) since \(x,y \in W_1\). Hence, we also have \(g = h\), which implies that \(p\) is injective.

It remains to show that \(p^{-1}\) is continuous. Identify \(W'\) with \(W_1\) and let \(\tilde p : \Gamma_Q(W') \rightarrow Q \times W_1\) be the polar decomposition of \(\Gamma_Q(W')\).
Let \(\varphi := \tilde p_2 \circ \tilde p \circ q: \Gamma_G(W) \rightarrow W_1\), where \(\tilde p_2 :  Q \times W_1 \rightarrow W_1\) is the projection onto the second component. Then, for all \(s = \exp_{G_\C}(ix)g \in \Gamma_G(W)\), where \(x \in W_1\), we have \(x = \varphi(s)\). Hence,
\[p^{-1} : \Gamma_G(W) \rightarrow W_1 \times \Gamma_G(W)^\times,\quad s \mapsto (\varphi(s),\exp_{G_\C}(-i\varphi(s))s),\]
is continuous and thus \(p\) is a homeomorphism.
\end{proof}

\begin{prop}
  \label{prop:Ol-homom}
  Let \(\g_1,\g_2\) be real Lie algebras and let \(W_j \subset \g_j\)(\(j=1,2\)) be a \(e^{\ad \g_j}\)-invariant wedge.
  Let \(\gamma : \g_1 \rightarrow \g_2\) be a homomorphism of Lie algebras with \(\gamma(W_1) \subset W_2\) and let \(\tilde\gamma : G_1 \rightarrow G_2\) be the corresponding homomorphism of 1-connected Lie groups.
  Furthermore, let \(\eta_{G_j}: G_j \rightarrow G_{j,\C}\) be the universal complexification of \(G_j\) (\(j=1,2\)).
  Then there exists a unique homomorphism \(\varphi : \Gamma_{G_1}(W_1) \rightarrow \Gamma_{G_2}(W_2)\) such that
  \begin{equation}
    \label{eq:Ol-homom}
    \varphi(\eta_{G_1}(g)\exp(ix)) = \eta_{G_2}(\tilde\gamma(g))\exp(i\gamma(x)), \quad \text{for } g \in G_1, x \in W_1.
  \end{equation}
  If \(W_1\) is generating, then \(\varphi\) is holomorphic on the interior of \(\Gamma_{G_1}(W_1)\).
\end{prop}
\begin{proof}
  By the universal property of the universal complexification of \(G_{1,\C}\), we obtain a unique holomorphic homomorphism \(\tilde\phi : G_{1,\C} \rightarrow G_{2,\C}\) such that \(\eta_{G_2} \circ \tilde\gamma = \tilde\phi \circ \eta_{G_1}\).
  Now we obtain \(\varphi\) by restricting \(\tilde\varphi\) to \(\Gamma_{G_1}(W_1)\).

  It remains to show that \(\varphi\) is unique.
  To this end, let \(\psi : \Gamma_{G_1}(W_1) \rightarrow \Gamma_{G_2}(W_2)\) be another homomorphism satisfying \eqref{eq:Ol-homom}.
  By Theorem \ref{thm:Ol-cplx}, we have \(\Gamma_{G_1}(W_1) = N_\C\eta_{G_1}(G_1)\exp(iW)\), where \(N_\C\) is the integral subgroup of \(G_{1,\C}\) with Lie algebra \(\fn_\C = H(W)_\C\).
  Thus, \(\varphi\lvert_{\eta_{G_1}(G_1)} = \psi\lvert_{\eta_{G_1}(G_1)}\) and \(\varphi\lvert_{\exp(iW_1)} = \psi\lvert_{\exp(iW_1)}\) follow immediately. Since \(\fn = H(W_1) \subset W_1\), we also have \(\L(\psi\lvert_{N_\C}) = \L(\varphi\lvert_{N_\C})\). This shows \(\varphi\lvert_{N_\C} = \psi\lvert_{N_\C}\) because \(N_\C\) is connected.
\end{proof}

We now turn to the real case.

\begin{theorem}
  \label{thm:Ol-real}
  Let \((G,\tau)\) be a 1-connected Lie group with Lie algebra \(\g = \fh \oplus \fq\). Let \(iW \subset \g^c \subset \g_\C\) be an \(e^{\ad \g^c}\)-invariant wedge and set \(C := W \cap \fq\) and \(\fn := H(C)\).
  Then the following holds:
  \begin{enumerate}
    \item The wedge \(C\) is \(e^{\ad \fh}\)-invariant and \(V := \fh + C\) is global in \(\g\). 
    \item The Lie subsemigroup \(\Gamma(C) := \oline{\la \exp_G(\fh + C) \ra}\) is \(*\)-invariant and \(\Gamma(C) := FH\exp(C)\), where \(F\) is the integral subgroup of \(G\) with \(\L(F) = [\fq,\fn] \oplus \fn\).
    \item The unit group \(\Gamma(C)^\times\) of \(\Gamma(C)\) is given by \(FH\).
    \item For any closed convex cone \(C_1 \subset \fq\) such that \(C = C_1 \oplus \fn\), the map
      \begin{equation}
        \label{eq:Ol-real-poldec}
        C_1 \times \Gamma(C)^\times \rightarrow \Gamma(C), \quad (y,g) \mapsto \exp(y)g,
      \end{equation}
      is a homeomorphism.
  \end{enumerate}
\end{theorem}
\begin{proof}
  It is clear that \(C\) is an \(e^{\ad \fh}\)-invariant wedge.
  In order to show that \(V\) is global, consider the inclusion map \(\iota^c : \fg \rightarrow \fg_\C^c\).
  Then \(\iota^c\) integrates to a homomorphism \(\tilde \iota^c : G \rightarrow G_\C^c\).
  Since the preimage of \(\g^c + W\), which is global in \(\g_\C^c\), under \(\iota^c\) is given by \(V\), we conclude with \cite[Prop.\ 1.41]{HN93} that \(V\) is global in \(\g\).
  Let \(\ff := [\fq,\fn] \oplus \fn \subset \g\).
  Then
  \[[\fh, \fn] = [\fh, H(W) \cap \fq] \subset \fn \quad \text{and} \quad [\fq, \fn] \subset \fh \cap H(iW), \quad [\fq, [\fq, \fn]] \subset \fn,\]
  implies that \(\ff\) is an ideal in \(\g\).
  The integral subgroup \(F \subset G\) with Lie algebra \(\ff\) is normal, hence it is 1-connected and closed, and the quotient Lie group \(Q := G/F\) is 1-connected as well (cf.\ \cite[Ch.\ XII, Thm.\ 1.2]{Ho65}).
  Let \(q: G \rightarrow Q\) be the quotient map.
  Then \(C_Q := \dd q(\1)(C) = C/\fn\) is weakly hyperbolic by Example \ref{ex:hyp-ell-sets}.
  Therefore, we obtain an Olshanski semigroup \(\Gamma(C_Q) \subset Q\) with \(q(\Gamma(C)) \subset \Gamma(C_Q) = H_Q\exp(C_Q)\), where \(H_Q = q(H)\). Thus, \(\Gamma(C) = FH\exp(C)\).

  By similar arguments as in the proof of Theorem \ref{thm:Ol-cplx}, we conclude that \(\Gamma(C)^\times = FH\) and that \eqref{eq:Ol-real-poldec} is a homeomorphism.
\end{proof}

We introduce the following notation: For an \(e^{\ad \g^c}\)-invariant wedge \(iW \subset \g_1^c\) and \(C := W \cap \fq\), we write \(\Gamma_1(C)\) for the semigroup we obtain from Theorem \ref{thm:Ol-real} when applied to the 1-connected Lie group \(G_1\) with Lie algebra \(\g_1 = [\fq,\fq] \oplus \fq\).

The semigroups \(\Gamma(C)\) and \(\Gamma_{G^c}(-iW)\) are related in the following way:
\begin{prop}
  \label{prop:Olsh-CW-homo}
  Let \((G,\tau)\) be a 1-connected Lie group with Lie algebra \(\g = \fh \oplus \fq\). Let \(iW \subset \g^c \subset \g_\C\) be an \(e^{\ad \g^c}\)-invariant wedge and let \(C := W \cap \fq\). Then there exists a continuous homomorphism \(\gamma : \Gamma(C) \rightarrow \Gamma_{G^c}(-iW)\) of \(*\)-semigroups with \(\gamma(C) \subset \Gamma_{G^c}(-iW) \cap \eta_G(G)\), where \(\eta_G : G \rightarrow G_\C^c\) is the universal complexification of \(G\), and which is equivariant with respect to the integrated adjoint action of the 1-connected Lie group of \(\tilde H\) with \(\L(\tilde H) = \fh\).
\end{prop}
\begin{proof}
  We obtain \(\gamma\) by integration of the inclusion map \(\g \hookrightarrow \g_\C^c\). The equivariance then follows from the \(e^{\ad \fh}\)-invariance of \(C\).
\end{proof}

\begin{lem}
  \label{lem:semgrp-dense-int-eq}
  Let \(G\) be a Lie group and \(S,T \subset G\) be open subsemigroups of \(G\) with \(T \subset S\). If \(T\) is dense in \(S\) and \(\1 \in \oline{S} = \oline T\), then \(S = T\).
\end{lem}
\begin{proof}
  We only have to show \(S \subset T\). Let \(s \in S\) and let \(U \subset G\) be an open neighborhood of \(s\) such that \(U \subset S\). Then \(V := Us^{-1}\) is an open \(\1\)-neighborhood, so that \(\1 \in \oline T\) implies that \(T^{-1} \cap V \neq \emptyset\), i.e.\ \(T^{-1}s \cap U \neq \emptyset\). Since \(T^{-1}s \cap U\) is open in \(S\) and \(T\) is dense in \(S\), we thus also have \(T \cap (T^{-1}s \cap S) \neq \emptyset\), i.e.\ there exists \(t,t' \in T\) such that \(t' = t^{-1}s\). Hence, we have \(S \subset TT \subset T\), which proves the claim.
\end{proof}

\begin{prop}
  \label{prop:Ol-open}
  \begin{enumerate}
    \item Let \(W \subset \g\) be an \(e^{\ad \g}\)-invariant generating wedge. Then \(\Gamma_G(W)^o = \Gamma_G(W)^\times\exp(iW^o)\). In particular, \(\Gamma_G(W)^\times\exp(iW^o)\) is a dense semigroup ideal in \(\Gamma_G(W)\).
    \item Let \(iW \subset \g^c\) be an \(e^{\ad \g^c}\)-invariant wedge and set \(C := W \cap \fq \subset \g\). If \(C\) is generating in \(\fq\), then \(\Gamma(C)^o = \Gamma(C)^\times \exp(C^o)\). In particular, \(\Gamma(C)^\times\exp(C^o)\) is a dense semigroup ideal in \(\Gamma(C)\).
  \end{enumerate}
\end{prop}
\begin{proof}
  (a)  Let \(\fn_\C\), \(N_\C\), \(\eta: G \rightarrow G_\C\), and \(q: \Gamma_G(W) \rightarrow \Gamma_Q(W')\) be defined as in Theorem \ref{thm:Ol-cplx}.  
  Let \(W_1\) be a closed convex cone such that \(W = W_1 \oplus \fn\), where \(\fn := H(W)\). 
  
  We first show that \(q^{-1}(\Gamma_Q({W'}^o)) = \Gamma_G(W^o)\). Let \(s \in \Gamma_G(W)\) such that \(q(s) \in \Gamma_Q(W'^o)\). By Theorem \ref{thm:Ol-cplx}, there exists a unique \(x \in W_1\) and \(g \in \Gamma_G(W)^\times\) such that \(s = \exp(ix)g\). By Lawson's Theorem \ref{thm:Lawson}, \(\exp(ix + \fn_\C)gN_\C \in \Gamma_Q(W'^o)\) implies that \(x + \fn_\C \in W'^o\) and, in particular, \(x \in W_1^o\). Thus, we have
  \[s \in \Gamma_G(W)^\times\exp(iW_1^o) \subset \Gamma_G(W)^\times\exp(iW^o).\]
  This shows that \(\Gamma_G(W)^\times\exp(iW^o)\) is an open semigroup ideal  because \(\Gamma_Q({W'}^o)\) is an open semigroup ideal in \(\Gamma_Q(W')\) by \cite[Thm.\ XI.1.12]{Ne00}. Since \(W\) is generating, the interior of \(W\) is dense in \(W\) (cf.\ \cite[Prop.\ 1.1(v)]{HN93}). Hence, \(\Gamma_G(W)^\times\exp(iW^o)\) is dense in \(\Gamma_G(W)\) by the polar decomposition (Theorem \ref{thm:Ol-cplx}). Now \(\Gamma_G(W)^o = \Gamma_G(W)^\times\exp(iW^o)\) follows from Lemma \ref{lem:semgrp-dense-int-eq}. The interior of a subsemigroup \(S\) of a Lie group with \(\1 \in \oline{S^o}\) is a dense semigroup ideal by \cite[Lem.\ 3.7]{HN93}

  (b) is proven in a similar way by using Theorem \ref{thm:Ol-real}.
\end{proof}

\begin{definition}
  \begin{inlinelist}
  \item Let \(W \subset \g\) be a non-empty open \(e^{\ad \g}\)-invariant convex cone. Then \(\oline W\) is an \(e^{\ad \g}\)-invariant wedge and we have \((\oline W)^o = W\). We define \(\Gamma_G(W) := \Gamma_G(\oline W)^o\).\\
  \item Let \(iW \subset \g^c\) be an \(e^{\ad \g^c}\)-invariant wedge such that \(C := (W \cap \fq)^o\) is non-empty, i.e. \(W \cap \fq\) is generating in \(\fq\). We define \(\Gamma(C) := \Gamma(\oline C)^o\).
  \end{inlinelist}
\end{definition}

\subsection{Extensions to representations of generalized Olshanski semigroups}

\begin{lem}
  \label{lem:rep-cone-cap-q}
  Let \((G,\tau)\) be a symmetric Lie group with Lie algebra \(\g = [\fq,\fq] \oplus \fq\) and \((\pi,\cH)\) be a strongly continuous unitary representation of \(G\). Then \((B(I_\pi) \cap \fq)^o = W_\pi \cap \fq\).
\end{lem}
\begin{proof}
  We first note that a closed convex cone has a non-empty interior if and only if it contains a basis of the surrounding vector space.
  Hence, if \((B(I_\pi) \cap \fq)^o = \emptyset\), then \(W_\pi \cap \fq = \emptyset\) because it does not contain a basis of \(\fq\).
  On the other hand, if \((B(I_\pi) \cap \fq)^o \neq \emptyset\), then \(\fq \subset B(I_\pi) - B(I_\pi)\), the \(\Ad(G)\)-invariance of \(B(I_\pi)\) (cf.\ \cite[Lem.\ X.1.3]{Ne00}), and the fact that \(\g\) is generated by \(\fq\) as a Lie algebra imply that
  \[\g = \spann \Ad(G).\fq \subset B(I_\pi) - B(I_\pi).\]
  Thus, we may assume for the rest of the proof that \((B(I_\pi) \cap \fq)^o \neq \emptyset\) and \(W_\pi \neq \emptyset\).

  Consider the semidirect product \(G_\tau := G \rtimes \{\1,\tau\}\) and define \(\sigma(g) := (\1,\tau)g(\1,\tau)\) for \(g \in G_\tau\). Then \(\sigma\) is an involutive automorphism of \(G_\tau\) with \(\sigma(g) = \tau(g)\) for \(g \in G \times \{\1\} \cong G\) which implies \(\L(\sigma) = \L(\tau)\). Let \((\pi^*,\cH^*)\) be the dual representation of \(\pi\) and set \(\pi_\tau^* := \pi^* \circ \tau\). Let \(\Phi : \cH \rightarrow \cH^*\) be the antiunitary operator defined by \(\Phi(v)(w) := \la v, w\ra\) and let
  \[J : \cH \oplus \cH^* \rightarrow \cH \oplus \cH^*, \quad (v,\lambda) \mapsto (\Phi^{-1}\lambda, \Phi v).\]
  Then \(J\) is an antiunitary involution and the representation \(\pi \oplus \pi_\tau^*\) of \(G\) extends to an antiunitary representation \((\rho, \cH \oplus \cH^*)\) of \(G_\tau\) such that \(\rho(\1,\tau) = J\) (cf.\ \cite[Lem.\ 2.10]{NO17}). In particular, we have \(\rho(\sigma(g)) = J\rho(g)J\) for \(g \in G_\tau\), which implies
  \[i\p\rho(\L(\tau)(x)) = i\p\rho(\L(\sigma)(x)) = -Ji\p\rho(x)J \quad \text{for } x \in \g.\]
  Hence \(B(I_\rho)\) is \((-\L(\tau))\)-invariant. Since \(\rho\lvert_G = \pi \oplus \pi_\tau^*\), we also have
  \[B(I_\rho) = B(I_\pi) \cap B(I_{\pi_\tau^*}) = B(I_\pi) \cap (-B(I_{\pi \circ \tau})) = B(I_\pi) \cap (-\L(\tau))B(I_\pi)\]
  and therefore \(B(I_\rho) \cap \fq = B(I_\pi) \cap \fq\). The argument at the beginning of the proof shows that \(W_\rho \neq \emptyset\). Furthermore, the \((-\L(\tau))\)-invariance of \(B(I_\rho)\) implies that \((B(I_\rho) \cap \fq)^o = W_\rho \cap \fq\) (cf.\ \cite[Prop.\ 1.6]{HN93}). Hence we have
  \[(B(I_\pi) \cap \fq)^o = (B(I_\rho) \cap \fq)^o = W_\rho \cap \fq = W_\pi \cap \fq.\qedhere\]
\end{proof}

\begin{prop}
  \label{prop:ol-q-inner-pts}
  Let \((\pi,\cH)\) be a strongly continuous non-degenerate \(*\)-representation of an open \(*\)-subsemigroup \(S\) of a symmetric Lie group \((G,\tau)\) with Lie algebra \(\g = \fh \oplus \fq\) and let \((\pi_1^c,\cH)\) be its analytic continuation to \(G_1^c\) {\rm(Theorem \ref{thm:si-integration})}. Then \(\L^o(S) \cap \fq \subset iB(I_{\pi_1^c})\) and, in particular, \((\L^o(S) \cap \fq)^o \subset iW_{\pi_1^c}\).
\end{prop}
\begin{proof}
  We first prove that \(\L^o(S) \cap \fq \subset iB(I_{\pi_1^c})\). Let \(x \in \L^o(S) \cap \fq\). Then, by Corollary \ref{cor:semigrp-rep-liecone-curve}, the curve
  \[\gamma : [0,\infty) \rightarrow B(\cH), \quad \gamma(0) := \id_\cH, \, \gamma(t) := \pi(\exp(tx)) \quad (t > 0),\]
  is a strongly continuous \(C^0\)-semigroup.
  Let \(A : \cD(A) \rightarrow \cH\) be its closed generator.
  Since \(\gamma\) leaves the dense subspace \(\cH^0 = \{\pi(f)\cH : f \in C_c^\infty(S)\}\) invariant, it is a core of \(A\) (cf.\ \cite[Thm.\ X.49]{ReSi75}).
  By a similar argument as in Proposition \ref{prop:semgrp-rep-curve} (a), we have \(\cD(A) \subset \cD(\cL^\pi_x)\) and \(\derat{0} \gamma(t)v = \cL^\pi_x v\) for all \(v \in \cD(A)\).
  Since \(\cH^0 \subset \cD(A)\), we have \(A\lvert_{\cH^0} = \cL^\pi_x\lvert_{\cH^0}\), hence \(A = \cL^\pi_x\) because \(\cH^0\) is also a core of \(\cL^\pi_x\) (cf.\ Remark \ref{rem:rkhs-ident-noncyclic}).
  In particular, \(\cL^\pi_x\) is the generator of a strongly continuous semigroup, which implies that \(\spec(\cL^\pi_x) \subset (-\infty,c)\) for some \(c \in \R\) (cf.\ \cite[Thm.\ X.47b]{ReSi75}).
  Since \(\cL^\pi_x = i\partial\pi_1^c(-ix)\), this shows \(-ix \in B(I_{\pi_1^c})\).

  Using Lemma \ref{lem:rep-cone-cap-q}, we conclude that
  \[(\L^o(S) \cap \fq)^o \subset (iB(I_\pi) \cap \fq)^o = iW_{\pi_1^c} \cap \fq \subset iW_{\pi_1^c}.\qedhere\]
\end{proof}

\begin{prop}
  \label{prop:Olsh-H-H1-rep}
  Let \((G,\tau)\) be a 1-connected symmetric Lie group with Lie algebra \(\g = \fh \oplus \fq\).
  \begin{itemize}
    \item Let \(H_1 \subset G\) be the integral subgroup of \(G\) with \(\L(H_1) = [\fq,\fq]\).
    \item Let \(H \subset G\) be the integral subgroup of \(G\) with \(\L(H) = \fh\).
    \item Let \(C \subset \fq\) be a non-empty \(e^{\ad \fh}\)-invariant weakly hyperbolic open convex cone.
    \item Let \(\tilde H\) be the 1-connected Lie group with \(\L(\tilde H) = \fh\).
  \end{itemize}
  Consider the semidirect product \(\tilde H \ltimes \Gamma_1(C)\), where \(\tilde H\) acts on \(\Gamma_1(C)\) by conjugation, and a strongly continuous representation
  \[\pi = (\pi^{\tilde H},\pi^S) : \tilde H \ltimes \Gamma_1(C) \rightarrow B(\cH).\]
  Let \(q_H : \tilde H \rightarrow H\) be the universal covering map of \(H\).
  If the representation \(\pi\lvert_{\tilde H \ltimes H_1}\) vanishes on the integral subgroup \(\Delta\) of \(\tilde H \ltimes H_1\) with \(\L(\Delta) = \{(x,-x) : x \in [\fq,\fq]\}\), then
  \[\tilde\pi : \Gamma(C) \rightarrow B(\cH), \quad q_H(h)\exp(y) \mapsto \pi^{\tilde H}(h)\pi^S(\exp(y)).\]
  is a well-defined strongly continuous representation of \(\Gamma(C)\).
\end{prop}
\begin{proof}
  We first recall from Lemma \ref{lem:semidir-prod-surj}(b) that the map
  \[\tilde H \ltimes H_1 \rightarrow H, \quad (\tilde h, h_1) \mapsto q(\tilde h)h_1\]
  is a surjective homomorphism of Lie groups whose kernel is \(\Delta\). Hence \((\tilde H \ltimes H_1)/\Delta \cong H\).
  Let now \(h,h' \in \tilde H\) and \(y \in C\) such that \(q_H(h)\exp(y) = q_H(h')\exp(y)\). Then \(h^{-1}h' \in \Delta\) implies that
  \[\pi^{\tilde H}(h)\pi^S(\exp(y)) = \pi^{\tilde H}(h')\pi^S(\exp(y)),\]
  which proves that \(\tilde \pi\) is well-defined. The multiplicativity of \(\tilde\pi\) follows from
  \begin{align*}
    \tilde\pi(q_H(h)\exp(y))\tilde\pi(q_H(h')\exp(y')) &= \pi^{\tilde H}(h)\pi^S(\exp(y))\pi^{\tilde H}(h')\pi^S(\exp(y'))\\
                                                       &= \pi^{\tilde H}(h)\pi^{\tilde H}(h')\pi^S(\exp(\Ad(h'^{-1})y))\pi^S(\exp(y')) \\
                                                       &= \pi^{\tilde H}(hh')\pi^S(\exp(\Ad(h'^{-1})y)\exp(y')) \\
                                                       &=\tilde\pi(q_H(hh')\exp(\Ad(h'^{-1})y)\exp(y'))\\
                                                       &= \tilde\pi(q_H(h)\exp(y)q_H(h')\exp(y')).
  \end{align*}
  The strong continuity of \(\tilde \pi\) follows from the fact that \(\Gamma(C) \cong ((\tilde H \ltimes H_1)/\Delta) \times C\) (cf.\ Theorem \ref{thm:Lawson}) and the strong continuity of \((\pi^{\tilde H},\pi^S)\) on \(\Gamma_1(C) \cong \tilde H \times H_1 \times C\).
\end{proof}

\begin{cor}
  \label{cor:Olsh-H-H1-rep}
  Let \((G,\tau),\g,H,H_1,\tilde H\) be defined as in {\rm Proposition \ref{prop:Olsh-H-H1-rep}}. Let \(iW \subset \g^c\) be an \(e^{\ad \g^c}\)-invariant wedge and set \(C := W \cap \fq \subset \g, \ff := [\fq,\fn] \oplus \fn, F := \la \exp_{G}(\ff)\ra,\) and \(\tilde H_F := \la \exp_{\tilde H}([\fq,\fn])\ra\). If \(C\) is generating in \(\fq\) and if the strongly continuous representation
  \[\pi = (\pi^{\tilde H},\pi^S) : \tilde H \ltimes \Gamma_1(C) \rightarrow B(\cH)\]
  vanishes on the integral subgroup \(\Delta\) of \(\tilde H \ltimes H\) with \(\L(\Delta) = \{(x,-x) : x \in [\fq,\fq]\}\) and \(\tilde H_F \subset \ker(\pi^{\tilde H}), F \subset \ker(\pi^S)\), then
  \[\tilde\pi : \Gamma(C)^o \rightarrow B(\cH), \quad q_H(h)\exp(y)f \mapsto \pi^{\tilde H}(h)\pi^S(\exp(y)).\]
  is a well-defined strongly continuous representation of \(\Gamma(C)^o\).
\end{cor}
\begin{proof}
  Recall from Proposition \ref{prop:Ol-open} that \(\Gamma(C)^o = \Gamma(C)^\times \exp(C^o)\). The Lie subalgebra \([\fq,\fn]\) is an ideal in \(\fh\) and \(\ff\) is an ideal in \(\g\). Hence, the Lie groups \(\tilde H/\tilde H_F\) and \(G/F\) are 1-connected (cf.\ \cite[Ch.\ XII, Thm.\ 1.2]{Ho65})and \(\Gamma_{Q_1}(C') \subset G/F\) is an Olshanski semigroup, where \(Q_1 := \la\exp_{G/F}([\fq,\fq]/[\fq,\fn])\ra\) and \(C' = C^o/\fn\).
  The representation \(\pi\) factors through a strongly continuous representation
  \[\pi_0 = (\pi_0^{\tilde H/\tilde H_F}, \pi_0^S) : (\tilde H/\tilde H_F) \ltimes \Gamma_{Q_1}(C') \rightarrow B(\cH)\]
  which satisfies the premises of Proposition \ref{prop:Olsh-H-H1-rep}. Hence, we obtain a strongly continuous representation \(\tilde\pi_0 : \Gamma_Q(C') \rightarrow B(\cH)\) with
  \[\tilde\pi_0(q_Q(h\tilde H_F)\exp(y + \fn)) = \pi_0^{\tilde H/\tilde H_F}(h\tilde H_F)\pi_0^S(\exp(y + \fn)) = \pi^{\tilde H}(h)\pi^S(\exp(y))\]
  for \(h \in \tilde H, y \in C^o\), where \(Q := \la\exp_{G/F}(\fh/[\fq,\fn])\ra\) and \(q_Q : \tilde H/\tilde H_F \rightarrow Q\) is a universal covering of \(Q\).
  Since we have a quotient map \(q: \Gamma(C^o) \rightarrow \Gamma_Q(C')\), we obtain a strongly continuous representation \(\tilde \pi := \tilde\pi_0 \circ q\) with
  \[\tilde\pi(q_H(h)\exp(y)) = \tilde\pi_0(q(q_H(h))\exp(y + \fn)) = \tilde\pi_0(q_Q(h\tilde H_F)\exp(y + \fn)) = \pi^{\tilde H}(h)\pi^S(\exp(y))\]
  for all \(h \in \tilde H, y \in C^o\).
\end{proof}

\begin{thm}
  \label{thm:Ol-semgrp-ext}
  Let \((G,\tau)\) be a connected symmetric Lie group with Lie algebra \(\g = \fh \oplus \fq\) and let \(S \subset G\) be an open \(*\)-subsemigroup, where \(g^* := \tau(g)^{-1}\) for \(g \in G\).
  Let \(\pi: S \rightarrow B(\cH)\) be a continuous non-degenerate \(*\)-representation of \(S\).
  If the interior of \(\L^o(S) \cap \fq\) in \(\fq\) is non-empty, then there exists a continuous non-degenerate \(*\)-representation \(\tilde \pi : \Gamma(C) \rightarrow B(\cH)\) of the \(*\)-semigroup \(\Gamma(C)\), where \(C := iW_{\pi_1^c} \cap \fq \supseteq (\L^o(S) \cap \fq)^o\) and \(\pi_1^c\) is the analytic continuation of \(\pi\) to \(G_1^c\) {\rm(cf.\ Theorem \ref{thm:semigrp-rep-int})}, such that
  \begin{equation}
    \label{eq:Ol-semgrp-ext-exp}
    \tilde\pi(\exp(x)) = \pi(\exp(x)), \quad \text{for } x \in (\L^o(S) \cap \fq)^o.
  \end{equation}
  In particular, if \(G\) is 1-connected, then \(\Gamma(C) \subset G\) and \(\tilde\pi\) is an extension of \(\pi\lvert_{S_0}\), where \(S_0\) is the subsemigroup generated by \(\exp_G((\L^o(S) \cap \fq)^o)\).
\end{thm}
\begin{proof}
  The open convex cone \(W_{\pi_1^c}\) is non-empty by Proposition \ref{prop:ol-q-inner-pts}. Let \(\fn = H(B(I_{\pi_1^c})) = \ker \dd \pi_1^c\) and \(N := \ker(\pi_1^c)\). Then \(\pi_1^c\) factors through a representation \(\tilde \pi_1^c\) of the 1-connected Lie group \(G_1^c/N\). Since \(W_{\tilde \pi_1^c} = W_{\pi_1^c}/\fn\) is weakly elliptic (cf.\ Example \ref{ex:hyp-ell-sets}(1)), we obtain by \cite[Thm.\ XI.2.3]{Ne00} a holomorphic Olshanski semigroup representation
  \[\pi_N: \Gamma_{G_1^c/N}(-W_{\tilde \pi_1^c}) \rightarrow B(\cH), \quad \pi_N(gN\exp(iy + \fn)) = \tilde\pi_1^c(gN)e^{-i\p\tilde\pi_1^c(iy + \fn)} = \pi_1^c(g)e^{-i\p\pi_1^c(iy)},\]
  which we pull back to a holomorphic representation
  \[\hat \pi: \Gamma_{G_1^c}(-W_{\pi_1^c}) \rightarrow B(\cH) \quad \text{with} \quad \hat\pi(g\exp(iy)) = \pi_1^c(g)e^{-i\p\pi_1^c(iy)}.\]
  Let \(\tilde G_1\) be the 1-connected Lie group with Lie algebra \(\g_1\) and let \(\Gamma_1(C)\) be the subsemigroup of \(\tilde G_1\) we obtain from Proposition \ref{prop:Ol-open}. Let \(\oline{\gamma}: \Gamma_1(\oline C) \rightarrow \Gamma_{G_1^c}(-\oline W_{\pi_1^c})\) be the homomorphism we obtain from Proposition \ref{prop:Olsh-CW-homo}. Then \(\oline \gamma\) restricts to a homomorphism \(\gamma: \Gamma_1(C) \rightarrow \Gamma_{G_1^c}(- W_{\pi_1^c})\) because of the construction of \(\oline \gamma\) and \((\oline C)^o \subset (i\oline{W_{\pi_1^c}})^o\). We define \(\tilde\pi_1 = \hat\pi \circ \gamma\).
  Let \(\tilde H\) be the 1-connected Lie group with \(\L(\tilde H) = \fh\) and let \((\pi^{\tilde H},\cH)\) be the unitary representation of \(\tilde H\) we obtain from Proposition \ref{prop:localrep-comm-rel}. Then we have
  \begin{align*}
    \pi^{\tilde H}(h)\tilde\pi_1(h_1\exp(y))\pi^{\tilde H}(h)^{-1} &= \pi^{\tilde H}(h)\pi_1^c(\gamma(h_1))e^{-i\p\pi_1^c(iy)}\pi^{\tilde H}(h)^{-1} \\
                                                                   &= \pi^{\tilde H}(h)\pi_1^c(\gamma(h_1))\pi^{\tilde H}(h)^{-1}e^{-i\p\pi_1^c(i\Ad(h)y)} \\
                                                                   &= \pi_1^c(h.\gamma(h_1))e^{-i\p\pi_1^c(i\Ad(h)y)} \\
                                                                   &= \pi_1^c(\gamma(h.h_1))e^{-i\p\pi_1^c(i\Ad(h)y)} \\
                                                                   &= \tilde\pi_1(h.(h_1\exp(y))),
  \end{align*}
  where \(\tilde H\) acts by the integrated adjoint representation.
  Note that \(\Gamma_1(C)\) is in fact invariant under the action of \(\tilde H\) because \(C\) is \(e^{\ad \fh}\)-invariant by Proposition \ref{prop:localrep-comm-rel}(b).
  The above computation shows that
  \[\nu: \tilde H \ltimes \Gamma_1(C) \rightarrow B(\cH), \quad \nu(h,s) = \pi^{\tilde H}(h)\tilde\pi_1(s)\]
is a representation of \(\tilde H \ltimes \Gamma_1(C)\).
By Theorem \ref{thm:localrep}, \(\nu\) satisfies the conditions of Corollary \ref{cor:Olsh-H-H1-rep}, so that we obtain a representation of the Olshanski semigroup \(\Gamma(C)\) by
\[\tilde\pi : \Gamma(C) \rightarrow B(\cH), \quad q_H(h)\exp(y) \rightarrow \pi^{\tilde H}(h)\tilde\pi_1(\exp(y)).\]
Because of \((\L^o(S) \cap \fq)^o \subset C\) (cf.\ Proposition \ref{prop:ol-q-inner-pts}), the representation \(\tilde\pi\) is an extension of \(\pi\) on \(\exp((\L^o(S) \cap \fq)^o) \subset \exp(C)\) in the sense of \eqref{eq:Ol-semgrp-ext-exp}.
\end{proof}

\begin{cor}
  \label{cor:Ol-semgrp-ancont}
  With the notation of {\rm Theorem \ref{thm:Ol-semgrp-ext}}, the analytic continuation of \((\pi,\cH)\) to \(G^c\) and the analytic continuation of \((\tilde\pi,\cH)\) to \(G^c\) coincide. 
\end{cor}
\begin{proof}
  Let \(q_G : \tilde G \rightarrow G\) be a universal covering of \(G\).
  We may assume that \(G\) is 1-connected because the analytic continuation of the representation \((\pi \circ q_G, \cH)\) of \(q_G^{-1}(S)\) to \(G^c\) coincides with the analytic continuation of \(\pi\) to \(G^c\). 
  Then the claim follows from \eqref{eq:Ol-semgrp-ext-exp} and Proposition \ref{prop:analytic-cont-undet}.
\end{proof}

\section{Examples}
\label{sec:examples}

In this section, we consider various examples of analytic continuations of \(*\)-representations of semigroups.

\begin{example}
  \label{ex:one-dim}
  The simplest non-trivial example is the one-dimensional case where \(G = \R,\, S = (c,\infty)\) for some \(c \geq 0\), and \(\tau = -\id_\R\). In this case, any strongly continuous non-degenerate representation \(\pi: S \rightarrow B(\cH)\) is selfadjoint and can be uniquely extended to a strongly continuous representation \(\tilde \pi: [0,\infty) \rightarrow B(\cH)\) with \(\tilde \pi(0) = \id_\cH\) (cf.\ Proposition \ref{prop:cont-op-semgrps}). The analytic continuation (Theorem \ref{thm:localrep-analytic-cont}) is given by
  \[\pi^c : \R \rightarrow \U(\cH), \quad t \mapsto e^{itA}\]
  where \(A\) is the infinitesimal generator of \(\tilde\pi\), i.e., \(\tilde\pi(t) = e^{tA}\) for \(t \geq 0\).
\end{example}

The main motivation for studying the analytic continuation problem in the first place comes from the field of reflection positivity which we mentioned in the introduction: A Hilbert space \(\cE\) is called a \emph{reflection positive Hilbert space} if there exists a unitary involution \(\theta \in \U(\cH)\) and a closed subspace \(\cE_+\) such that
\[\la \xi, \xi\ra_\theta := \la \xi, \theta \xi\ra \geq 0 \quad \text{for all } \xi \in \cE_+.\]
The space \(\cE_+\) is called \emph{\(\theta\)-positive}. We then obtain a scalar product on the quotient
\[\cE_+/\cN \quad \text{by} \quad \cN := \{\eta \in \cE_+ : \la \eta, \theta \eta\ra = 0\},\]
via \(\|\hat{v}\|_{\hat{\cE}} := \sqrt{\langle v, \theta v\rangle}\), where \(\hat{v}\) denotes the image of \(v \in \cE_+\) under the canonical quotient map \(\cE_+ \rightarrow \cE_+ / \cN\).
Completing \(\cE_+ / \cN\) with respect to this scalar product leads to a Hilbert space \(\hat{\cE}\). We write reflection positive Hilbert spaces as triples \((\cE,\cE_+,\theta)\).

Consider now a symmetric Lie group \((G,\tau)\) and a unitary representation \((\pi,\cE)\) of the semidirect product \(G_\tau = G \rtimes \{\1,\tau\}\) on the reflection positive Hilbert space \((\cE,\cE_+,\theta)\) with \(\pi(\tau) = \theta\). Then the restriction of \(\pi\) to the \(*\)-semigroup
\[S := \{g \in G : \pi(g)\cE_+ \subset \cE_+\}\]
factors through a strongly continuous contraction representation of \(S\) on \(\hat\cE\) (cf.\ \cite[Prop.\ 3.3.3]{NO18}). If \(S\) has a non-empty interior, then we can apply the analytic continuation Theorems \ref{thm:si-integration} and \ref{thm:localrep} to obtain a strongly continuous unitary representation of \(G_1^c\) or \(G^c\) on \(\hat\cE\).

\begin{example}
  Let \(\cH\) be a complex Hilbert space. A \emph{standard subspace} \(V\) is a closed real subspace of \(\cH\) such that \(V + iV\) is dense in \(\cH\) and \(V \cap iV = \{0\}\). The set of standard subspaces of \(\cH\) is denoted by \(\Stand(\cH)\). There is a one-to-one correspondence between \(\Stand(\cH)\) and the set of \emph{modular objects}, which consists of pairs \((\Delta,J)\), where \(\Delta\) is a positive operator on \(\cH\) and \(J\) is an antiunitary involution on \(\cH\) satisfying \(J\Delta J = \Delta^{-1}\). For \(V \in \Stand(\cH)\), the corresponding modular pair \((\Delta_V,J_V)\) is obtained by taking the polar decomposition of the conjugation operator
  \[S : V + iV \rightarrow \cH, \quad x + iy \mapsto x - iy,\]
  i.e.\ \(S = J\Delta^{1/2}\) (cf.\ \cite{Lo08}).

  Let \(V \in \Stand(\cH)\). Then we have the relation
  \[\la v, Jv \ra = \la v, \Delta^{1/2}v\ra \geq 0 \quad \text{for all } v \in V.\]
  The triple \((\cH, V, J)\) thus becomes a real reflection positive Hilbert space. The subspace \(\cN = \linebreak\{\eta \in V : \la \eta, J \eta\ra = 0\}\) is trivial because
  \[\la v, J v\ra = \|\Delta^{1/4}v\|^2 \quad \text{for } v \in V\]
  and \(\Delta^{1/4}\) is injective. The Hilbert space \(\hat V\) corresponding to \((\cH,V,J)\) can be identified with
  \[\oline{\Delta^{1/4}V} = \cH^J := \{v \in \cH : Jv = v\}.\]
  Let now \((G,\varepsilon)\) be a \emph{graded Lie group}, i.e.\ \(G\) is a Lie group and \(\varepsilon : G \rightarrow \{-1,1\}\) is a continuous group homomorphism. Furthermore, let \((\pi,\cH)\) be a strongly continuous antiunitary representation of \(G\), i.e.\ \(\pi(g)\) is linear if and only if \(\varepsilon(g) = 1\). Let \(\gamma \in \Hom(\R^\times,G)\) such that \(\pi(\gamma(-1))\) is antiunitary. By setting
  \[J := \pi(\gamma(-1)) \quad \text{and} \quad \Delta^{-it/2\pi} := \pi(\gamma(e^t)), \quad t \in \R,\]
  we obtain a modular pair \((\Delta,J)\) and thus a standard subspace \(V_\gamma\). The passage from \(\gamma\) to \(V_\gamma\) for a fixed antiunitary representation \(\pi\) is known as the Brunetti--Guido--Longo-map (cf.\ \cite[Cor.\ 2.4]{Ne17}). We define an involutive automorphism on \(G\) by \(\tau(g) := \gamma(-1)g\gamma(-1)\). Through the procedure outlined above, we then obtain a strongly continuous \(*\)-representation \((\hat\pi, \cH^J)\) of the semigroup
  \[S_V := \{g \in G_1 : \pi(g)V \subset V\}.\]
  The group \(G_1 := \varepsilon^{-1}(\{1\})\) acts on the set of standard subspaces by the representation \(\pi\) and the semigroup \(S_V\) contains all information about the inclusions of standard subspaces on the orbit \(\pi(G_1)V\). If \(S_V^o\) is non-empty  and \(\L^o(S_V^o) \cap \fq\) has inner points, then the semigroup \(\Gamma(C)\) we obtain from Theorem \ref{thm:Ol-semgrp-ext} provides additional insight about the original semigroup \(S_V\).
\end{example}

Let \((G,\tau)\) be a symmetric Lie group with Lie algebra \(\g = \fh \oplus \fq\) and suppose  that \(\g\) is 3-graded, i.e.\ there exists a decomposition \(\fq = \fq_- \oplus \fq_+\) of \(\fq\) into abelian subalgebras such that \(\g = \fq_- \oplus \fh \oplus \fq_+\). Such decompositions appear for instance in the theory of non-Riemannian semisimple symmetric spaces (cf.\ \cite{HO96}). For an open convex cone \(C = C_- \oplus C_+ \subset \fq_- \oplus \fq_+\), the semigroup \(S_C := \oline{\la \exp(C_-)\exp(C_+)\ra}\) is \(*\)-invariant and \(C_-,C_+ \subset \L(S_C) \cap \fq\). In particular, there are cases for which \(S_C \neq H\exp(C)\):

\begin{lemma}
  \label{lem:ex-global-semgrp}
  Let \((G,\tau)\) be a symmetric Lie group with Lie algebra \(\g = \fh \oplus \fq\). Let \(W \subset \g\) be an \(e^{\ad \g}\)-invariant wedge and set \(C := W \cap \fq\). Furthermore, let \(W_1 \subsetneq \fh \oplus C\) be a closed convex cone such that
  \begin{enumerate}
    \item \(W_1 \cap (\fh \oplus H(C))\subset H(W_1)\),
    \item \((-\L(\tau))(W_1) = W_1\), and
    \item \(H(W_1)\) is global.
  \end{enumerate}
  Then \(S_{W_1} = \oline{\la \exp(W_1)\ra}\) is a \(*\)-invariant subsemigroup of \(G\) with \(\L(S_{W_1}) = W_1 \neq \fh + C = \L(\Gamma(C))\).
\end{lemma}
\begin{proof}
  By Theorem \ref{thm:Ol-real}, the wedge \(W' := \fh + C\) is global in \(\g\) and \(S_{W'} = \Gamma(C)\). Hence, by \cite[Prop.\ 1.37]{HN93}, the wedge \(W_1\) is global in \(G\) and thus we have \(\L(S_{W_1}) = W_1\).
\end{proof}

\begin{example}
  \label{ex:cayley-type}
  Let \((G,\tau)\) be a symmetric Lie group and suppose that its Lie algebra \(\g = \fq_- \oplus \fh \oplus \fq_+\) is 3-graded. Let \(W\) and \(C\) be as in Lemma \ref{lem:ex-global-semgrp} with \(C = C_- \oplus C_+ \subset \fq_- \oplus \fq_+\). If \(C\) is pointed, then the conditions of Lemma \ref{lem:ex-global-semgrp} are satisfied for the cone \(W_1 := C\), so that 
  \[S_C = \oline{\la\exp(C_-)\exp(C_+)\ra} = \oline{\la\exp(C)\ra} \neq \Gamma(C).\] 

  We give a concrete example: Let \(G = \tilde\SL_2(\R)\) be the universal covering group of \(\SL_2(\R)\) and let \(\tau \in \Aut(G)\) be the integral of the automorphism
  \[\L(\tau): \fsl_2(\R) \rightarrow \fsl_2(\R), \quad \pmat{x & y \\ z & -x} \mapsto \pmat{x & -y \\ -z & -x}.\]
  Then
  \[\fh = \R \pmat{1 & 0 \\ 0 & -1}, \quad \fq_- = \R\pmat{0 & 0 \\ 1 & 0}, \quad \fq_+ = \R\pmat{0 & 1 \\ 0 & 0}.\]
  Let \(\omega(x,y) := x_1y_2 - x_2y_1\) be the standard symplectic form on \(\R^2\). Since the cone
  \[W := \left\{A = \pmat{x & y \\ z & -x} \in \fsl_2(\R) : (\forall v \in \R^2)\, \omega(Av,v) \geq 0\right\}\]
  is \(e^{\ad \g}\)-invariant, the convex cone \(C := W \cap \fq = \R_{\geq 0}\pmat{0 & 1 \\ 0 & 0} + \R_{\geq 0}\pmat{0 & 0 \\ -1 & 0}\) satisfies the premises of Lemma \ref{lem:ex-global-semgrp}. Furthermore, \(C\) is pointed and weakly hyperbolic, so that Lawson's Theorem \ref{thm:Lawson} implies that \(H\exp(C^o)\) is an Olshanski semigroup. Since \(C\) is global in \(G\), we have
  \[S_C = \oline{\la\exp(C_-)\exp(C_+)\ra} = \oline{\la\exp(C)\ra}.\]
  Let \(S := S_C^o\) and let \(\pi : S \rightarrow B(\cH)\) be a continuous non-degenerate \(*\)-representation of \(S\) on a complex Hilbert space \(\cH\). Then, by Theorem \ref{thm:Ol-semgrp-ext}, there exists a unique extension of \(\pi\) to a strongly continuous non-degenerate representation \(\tilde\pi: \Gamma(C) \rightarrow B(\cH)\). The analytic continuations to \(G^c \cong G\) of \(\pi\) and \(\tilde\pi\) coincide by Corollary \ref{cor:Ol-semgrp-ancont}. 
\end{example}

\appendix
\section{Differentiable vectors and generators}
\label{sec:diffvectors}
Let \(\cH\) be a Hilbert space, \(G\) be a finite-dimensional Lie group, and \((\pi,\cH)\) be a strongly continuous unitary representation of \(G\). For \(k \in \N_0\), we denote by \(\cH^k\) the space of vectors \(v \in \cH\) such that the orbit map \(\pi^v : G \rightarrow \cH\) is \(C^k\). We denote the space of smooth vectors by \(\cH^\infty\).

For \(x \in \g = \L(G)\), we define \(\cD_x := \{v \in \cH : \derat{0} \pi(\exp(tx))v \,\text{ exists}\}\) and
\[\p\pi(x) : \cD_x \rightarrow \cH, \quad \p\pi(x)v := \derat{0} \pi(\exp(tx))v.\]
By restricting to \(\cH^\infty\), we obtain a Lie algebra representation
\[\dd\pi : \g \rightarrow \End(\cH^\infty), \quad x \mapsto \dd\pi(x) := \p\pi(x)\lvert_{\cH^\infty}\]
by essentially skew-adjoint operators and we have
\begin{equation}
  \label{eq:diffvectors-hinf-dinf}
  \cH^\infty = \cD^\infty = \bigcap_{n \in \N,\, x_i \in \g} \cD(\p\pi(x_1),\ldots,\p\pi(x_n))
\end{equation}
(cf.\ \cite[Lem.\ 3.4]{Ne10}). Since \(\cH^\infty\) is \(\pi\)-invariant and dense in \(\cH\) (cf.\ \cite[Prop.\ 4.4.1.1]{Wa72}), we have \(\oline{\dd\pi(x)} = \partial\pi(x)\) by Stone's theorem.

\begin{prop}
  \label{prop:smoothvectors-dual}
  Let \(\cH^{-\infty}\) be the space of antilinear functionals on \(\cH^\infty\). We consider \(\cH\) as a subspace of \(\cH^{-\infty}\) by setting \(v(w) := \la w, v \ra\) for \(v \in \cH, w \in \cH^\infty\). Let
  \begin{equation}
    \label{eq:diffvectors-dualrep}
  \dd\pi^{-\infty} : \g \rightarrow \End(\cH^{-\infty}), \quad (\dd\pi^{-\infty}(x)\lambda)(w) := -\lambda(\dd\pi(x)w)
  \end{equation}
  be the \emph{dual representation} of \(\g\). Then we have for all \(x \in \g\):
  \[\cD(\p\pi(x)) = \{v \in \cH : \dd\pi^{-\infty}(x)v \in \cH\}.\]
  In particular, all \(v \in \cD(\p\pi(x))\) satisfy \(\p\pi(x)v = \dd\pi^{-\infty}(x)v\) as elements of \(\cH^{-\infty}\).
\end{prop}
\begin{proof}
  Recall that \(\p\pi(x) = - (\dd\pi(x))^* = \oline{\dd\pi(x)}\). Let \(v \in \cH\) such that \(u := \dd\pi^{-\infty}(x)v \in \cH\). This is equivalent to
  \[ - \la \dd\pi(x)w, v \ra = \la w, u \ra \quad \text{for all } w \in \cH^\infty,\]
  and this implies that \(v \in \cD((\dd\pi(x)\lvert_{\cH^\infty})^*) = \cD(\p\pi(x))\) and \(\p\pi(x)v = u\).

  Conversely, let \(v \in \cD(\p\pi(x))\). Then we have
  \[(\dd\pi^{-\infty}(x)v)(w) = -\la \dd\pi(x)w, v \ra = \la w, \p\pi(x)v\ra = (\p\pi(x)v)(w) \quad \text{for all } w \in \cH^\infty.\]
  In particular, we have \(\dd\pi^{-\infty}(x)v = \p\pi(x)v\) as elements of \(\cH^{-\infty}\).
\end{proof}

\begin{cor}
  \label{cor:smoothvectors-invsubspace}
  Let \(V \subset \cH\) be a subspace. Then
  \[\g_V := \{x \in \g : V \subset \cD(\p\pi(x)) \text{ and } \p\pi(x)V \subset V\}\]
  is a Lie subalgebra of \(\g\) and
  \[\g_V \rightarrow \End(V), \quad x \mapsto \p\pi(x)\lvert_V,\]
  is a Lie algebra representation of \(\g_V\).
\end{cor}
\begin{proof}
  We consider \(V\) as a subspace of \(\cH^{-\infty}\). Since the dual representation \eqref{eq:diffvectors-dualrep} is a Lie algebra representation, the subspace
  \[\tilde\g_V := \{x \in \g : \dd\pi^{-\infty}(x)V \subset V\}\]
  is a Lie subalgebra of \(\g\). If \(x \in \tilde\g_V\), then \(\dd\pi^{-\infty}(x)V \subset V \subset \cH\) and Proposition \ref{prop:smoothvectors-dual} imply that \(V \subset \cD(\partial\pi(x))\) and \(\partial\pi(x)V = \dd\pi^{-\infty}(x)V \subset V\), i.e.\ \(x \in \g_V\). The converse inclusion \(\g_V \subset \tilde\g_V\) also follows from Proposition \ref{prop:smoothvectors-dual}. Hence \(\g_V = \tilde\g_V\) is a Lie subalgebra and \(\partial\pi(x)\lvert_V = \dd\pi^{-\infty}(x)\lvert_V\) shows that \(\partial\pi\) restricts to Lie algebra representation of \(\g_V\) on \(V\). 
\end{proof}

Let \(E \subset \g\) be a set of generators of the Lie algebra \(\g\). Then the infinitesimal generators belonging to elements of \(E\) already determine the set of smooth vectors:

\begin{prop}
  \label{prop:smoothvec-generators}
  Suppose that the subset \(E \subset \g\) generates \(\g\) as a Lie algebra. Then we have
  \[\cH^\infty = \bigcap_{n \in \N,\, x_i \in E} \cD(\p\pi(x_1),\ldots,\p\pi(x_n)).\]
\end{prop}
\begin{proof}
  Let \(E^\infty := \bigcap_{n \in \N,\, x_i \in E} \cD(\p\pi(x_1),\ldots,\p\pi(x_n))\). 
  By \eqref{eq:diffvectors-hinf-dinf}, \(\cH^\infty\) is contained in \(E^\infty\).
  In order to prove the converse inclusion, consider the set
  \[\fk := \g_{E^\infty} = \{x \in \g : E^\infty \subset \cD(\p\pi(x)) \text{ and } \p\pi(x)E^\infty \subset E^\infty\}.\]
  By Corollary \ref{cor:smoothvectors-invsubspace}, \(\fk\) is a Lie subalgebra of \(\g\) and by the definition of \(E^\infty\), we have \(E \subset \fk\).
  Hence, \(\fk = \g\), which means that \(E^\infty \subset \cD(\p\pi(x))\) and \(\p\pi(x)E^\infty \subset E^\infty\) for all \(x \in \g\), i.e.\ \(E^\infty \subset \cD^\infty = \cH^\infty\).
\end{proof}

\section{Holomorphic functions}

Throughout this section, let \(\cH\) be a complex Hilbert space.

\begin{lem}
  \label{lem:holom-exp-prod}
  Let \(A : \cD(A) \rightarrow \cH\) be a selfadjoint operator on \(\cH\) and let \(C \in B(\cH)\) such that \(e^{tA}C\) is a bounded operator on \(\cH\) for \(t = a,b \in \R\), where \(a < b\). Then the map
  \[F: \cS_{a,b} := \{z \in \C : a < \Re z < b\} \rightarrow B(\cH), \quad z \mapsto e^{zA}C,\]
  is holomorphic.
\end{lem}
\begin{proof}
  Let \(P\) be the spectral measure corresponding to \(A\) and set \(P^{v,w}(E) := \la P(E)v,w\ra\), where \(v,w \in \cH\) and \(E \subset \R\) is Borel-measurable. Then we have
  \[\la v, Aw \ra = \int_{-\infty}^\infty \lambda \,dP^{v,w}(\lambda) \quad \text{for } v,w \in \cD(A).\]
  According to \cite[Lem.\ 2.1]{NSZ17}, the boundedness of \(e^{tA}C\) is equivalent to \(C(\cH) \subset \cD(e^{tA})\) for \(t \in \R\). Using this spectral integral representation of \(A\), we see that this and the assumption imply that \(e^{tA}C\) is bounded for \(t \in [a,b]\). Hence, \(e^{zA}C\) is bounded for \(z \in \cS_{a,b}\) because \(e^{itA}\) is unitary for \(t \in \R\).
  It remains to show that \(F\) is holomorphic: The function
  \[f: \cS_{a,b} \rightarrow \R, \quad z \mapsto \|e^{zA}C\| = \|e^{(\Re z)A}C\| = \sup \{\|e^{zA}C\xi\|, \xi \in \cH, \|\xi\| \leq 1\}\]
  is a plurisubharmonic function because
  \[\cS_{a,b} \ni z \mapsto \|e^{\Re z}C\xi\| = \sqrt{\int_{-\infty}^\infty e^{(\Re z)\lambda}\, dP^{C\xi,\xi}(\lambda)}\]
  is plurisubharmonic for all \(\xi \in \cH\) and a supremum of a set of plurisubharmonic functions is again plurisubharmonic \cite[Lem.\ XIII 4.4(b)]{Ne00}. Since \(f(z)\) does not depend on the imaginary part of \(z\) for each \(z \in \cS_{a,b}\), \cite[Ex.\ XIII 4.3(c)]{Ne00} implies that \(f\) is convex. Hence, \(F\) is locally bounded.
  Furthermore, for each \(v,w \in \cD(A)\), the map
  \[\cS_{a,b} \rightarrow \C, \quad z \mapsto \la v, F(z)w\ra = \la v, e^{zA}Cw \ra = \int_{-\infty}^\infty e^{z\lambda}\,dP^{v,Cw}(\lambda)\]
  is holomorphic (cf.\ \cite[Prop.\ V.4.6]{Ne00}). Thus, \(F\) is holomorphic by \cite[Cor.\ A.III.5]{Ne00}.
\end{proof}

\subsection*{Acknowledgement}
I would like to thank Karl-Hermann Neeb for all the helpful discussions during my work on this topic and proof-reading of earlier versions of this article.


\end{document}